\documentclass[a4paper,11pt,oneside]{amsart}
\pdfoutput=1

\usepackage[utf8]{inputenc}
\usepackage{bm}
\usepackage{mathtools,amssymb}
\usepackage{esint}
\usepackage{hyperref}
\hypersetup{
    colorlinks=true,
    linkcolor=black,
    filecolor=black,      
    urlcolor=cyan,
    citecolor=black
    }
\usepackage{tikz}
\usepackage{dsfont}
\usepackage{relsize}
\usepackage{url}
\usepackage{xcolor}
\usepackage{graphicx}
\usepackage{mathrsfs}
\usepackage[shortlabels]{enumitem}
\usepackage{lineno}
\usepackage{amsmath}
\usepackage{enumitem}
\usepackage{amsthm} 
\usepackage{verbatim}
\usepackage{dsfont}

\allowdisplaybreaks

\mathtoolsset{showonlyrefs}

\graphicspath{{images/}}

\newtheorem{theorem}{Theorem}[section]
\newtheorem{lemma}[theorem]{Lemma}
\newtheorem{definition}[theorem]{Definition}
\newtheorem{corollary}[theorem]{Corollary}

\newtheorem{proposition}[theorem]{Proposition}

\newtheorem{remark}[theorem]{Remark}

\title[The inverse fractional conductivity problem]{The global inverse fractional conductivity problem}

\author{Giovanni Covi}
\address{Institut fur Angewandte Mathematik, Ruprecht-Karls-Universit\"at Heidelberg, Im Neuenheimer Feld 205, 69120 Heidelberg, Germany}
\email{giovanni.covi@uni-heidelberg.de}
\author{Jesse Railo}
\address{Department of Pure Mathematics and Mathematical Statistics, University of
Cambridge, Cambridge CB3 0WB, UK}
\email{jr891@cam.ac.uk}
\author{Philipp Zimmermann}
\address{Department of Mathematics, ETH Zurich, Z\"urich, Switzerland}
\email{philipp.zimmermann@math.ethz.ch}
\date{\today}

\newcommand{\C}{{\mathbb C}}
\newcommand{\R}{{\mathbb R}}
\newcommand{\Z}{{\mathbb Z}}
\newcommand{\N}{{\mathbb N}}

\newcommand{\schwartz}{\mathscr{S}}

\newcommand{\tempered}{\mathscr{S}^{\prime}}

\newcommand{\fourier}{\mathcal{F}}
\newcommand{\ifourier}{\mathcal{F}^{-1}}
\newcommand{\vev}[1]{\left\langle#1\right\rangle}

\newcommand{\norm}[1]{\lVert #1 \rVert}
\newcommand{\abs}[1]{\left\lvert #1 \right\rvert}
\newcommand{\ip}[2]{\left\langle #1,#2 \right\rangle}
\DeclareMathOperator{\Div}{div} 
\DeclareMathOperator{\supp}{supp} 
\DeclareMathOperator{\dist}{dist} 

\begin{document}

\maketitle
\begin{abstract}
We prove \emph{global} uniqueness for an inverse problem for the fractional conductivity equation on domains that are bounded in one direction. The conductivities are assumed to be isotropic and nontrivial in the exterior of the domain, while the data is given in the form of partial Dirichlet-to-Neumann (DN) maps measured in nondisjoint open subsets of the exterior. This can be seen as the fractional counterpart of the classical inverse conductivity problem. 
The proof is based on a unique continuation property (UCP) for the DN maps and an exterior determination method from the partial exterior DN maps. This is analogous to the classical boundary determination method by Kohn and Vogelius.
The most important technical novelty is the construction of sequences of special solutions to the fractional conductivity equation whose Dirichlet energies in the limit can be concentrated at any given point in the exterior. This is achieved independently of the UCP and despite the nonlocality of the equation. Due to the recent counterexamples by the last two authors, our results almost completely characterize uniqueness for the inverse fractional conductivity problem with partial data for isotropic global conductivities.
\end{abstract}


\section{Introduction}

In this article we study the uniqueness, reconstruction and stability properties of the inverse problem for the fractional conductivity equation
\begin{equation}
\label{eq:conductivity intro}
        \begin{split}
            \mbox{div}_s(\Theta_{\gamma}\cdot\nabla^s u)&= 0\quad\text{in}\quad\Omega,\\
            u&= f\quad\text{in}\quad\Omega_e,
        \end{split}
\end{equation}
for a possibly unbounded open set $\Omega\subset\R^n$ which is bounded in one direction and $0<s<\min(1,n/2)$. Here $\gamma \in L^\infty(\R^n)$ is a positive conductivity, and $\Theta_{\gamma}(x,y)\vcentcolon =\gamma^{1/2}(x)\gamma^{1/2}(y)\mathbf{1}_{n\times n}$ for $x,y\in\R^n$ denotes the conductivity matrix. The fractional gradient of order $s$ is the bounded linear operator $\nabla^s\colon H^s(\R^n)\to L^2(\R^{2n};\R^n)$ given by
    \[
        \nabla^su(x,y)=\sqrt{\frac{C_{n,s}}{2}}\frac{u(x)-u(y)}{|x-y|^{n/2+s+1}}(x-y),
\]
where $C_{n,s}>0$ is a constant. This operator appears in the study of nonlocal diffusion, weighted long jump random walk models, and is naturally associated with the $L^2$ Gagliardo seminorm. We denote the adjoint of $\nabla^s$ by div$_s$. Such operators were first introduced in the fractional calculus theory developed in~\cite{NonlocDiffusion}, and later studied in~\cite{C20,RZ2022unboundedFracCald}. They enjoy good mapping properties and verify the equality div$_s\nabla^s = (-\Delta)^s$. We use the notation $m_{\gamma}\vcentcolon =\gamma^{1/2}-1$ for the background deviation and $\Lambda_\gamma$ for the exterior DN map associated to \eqref{eq:conductivity intro}. 

The fractional conductivity operator $\mbox{div}_s(\Theta_{\gamma}\cdot\nabla^s)$ in \eqref{eq:conductivity intro} is an isotropic special case of more general nonlocal operators given weakly by bilinear forms of the type:
\begin{equation}
    B_K(f,g) =\int_{\R^{2n}} K(x,y)\nabla^s f \cdot \nabla^s g\,dxdy,   \quad K\in L^{\infty}(\R^{2n};\R^{n\times n}).
\end{equation}
In the isotropic case $K(x,y) =k(x,y)\mathbf{1}_{n\times n}$ for some $k\in L^{\infty}(\R^{2n})$, this reduces into the form
\begin{equation}\label{eq:generalNonlocalOperators}
    B_K(f,g)=\frac{C_{n,s}}{2}\int_{\R^{2n}} \frac{k(x,y)}{\abs{x-y}^{n+2s}} (f(x)-f(y))(g(x)-g(y))\,dxdy.
\end{equation} 
See for instance the survey~\cite{RosOton16-NonlocEllipticSurvey} of isotropic nonlocal elliptic equations on bounded domains.

We motivate this research problem and discuss how it compares to the classical Calderón problem in details in Section~\ref{sec: motivation}. We give a brief outline of the mathematical structures behind our work in Section~\ref{sec:outlineofProofs}. These discussions highlight many aspects of the rich and fascinating mathematical theory of the \emph{global Calderón problem for the fractional conductivity equation}. Our work almost completely characterizes global uniqueness for the related partial data inverse problems. We rely on many ideas from the recent works~\cite{C20,GSU20,RS17,RZ2022FracCondCounter,RZ2022unboundedFracCald} and the classical work~\cite{KohnVogelius} of Kohn and Vogelius. On the background, somewhat transparently in our article, one utilizes Carleman estimates for the Caffarelli--Silvestre extensions~\cite{CS07,Ru15} via the unique continuation property (UCP) of the fractional Laplacians (see the seminal work~\cite{GSU20} by Ghosh, Salo and Uhlmann for details).

Both the UCP and a comprehensive analysis of many different aspects of the fractional conductivity equation in~\cite{C20,RZ2022FracCondCounter,RZ2022unboundedFracCald} are essential in our proof. In the present work, we have achieved a \emph{global} uniqueness result for an inverse problem of the nonlocal variable coefficient operators~\eqref{eq:generalNonlocalOperators} with \emph{nontrivial coefficients in the whole Euclidean space}. This interesting result is formulated in Theorem~\ref{thm: Global uniqueness}. A construction of sequences of special solutions whose Dirichlet energies can be concentrated in the limit, despite the nonlocal nature of the studied operators, is an indispensable part of our theory. Though possible generalizations, sharper regularity assumptions, global stability, reconstruction, and many other natural questions remain viable research directions, the analysis of the model problem~\eqref{eq:conductivity intro} may drive research towards new objectives beyond a standard application of the UCP and enriches the already vast mathematical theory of nonlocal inverse problems. 

The theory of nonlocal inverse problems has been widely developed in the last years~\cite{S17}, but it has focused on the determination of additive perturbations in $\Omega$ or operators with trivial a priori known coefficients in the exterior $\Omega_e$ (see e.g.~\cite{CMRU20,C21,ghosh2021calderon,LL-fractional-semilinear-problems,RZ2022unboundedFracCald} and the references therein). The recent work~\cite{feizmohammadiEtAl2021fractional}, solving a geometric inverse problem for the fractional spectral Laplacian on closed Riemannian manifolds, is a notable exception where the nonlocal operator is assumed to be nontrivial in the whole space. The data in~\cite{feizmohammadiEtAl2021fractional} is given by local source-to-solution maps of globally defined fractional (spectral) Poisson's equations, and the measurements take place in a fixed support of controllable sources. Analogously, in our setting the measurements take place in a fixed support of controllable exterior values. However, these two inverse problems are quite different, as $\R^n$ is an open manifold, the fractional spectral Laplacians depend on the set $\Omega$ (whereas the Fourier multiplier $(-\Delta)^s$ or nonlocal operators with the weak form~\eqref{eq:generalNonlocalOperators} do not), and exterior value and interior source problems model different geometric measurement settings.

We describe our main results next. We have obtained the following results dealing with the invariance of data (for conductivities which a priori agree on the sets where measurements are performed) and exterior determination:

\begin{theorem}[UCP of the exterior DN map and invariance of data]
\label{thm: ucp of dn map}
    Let $\Omega\subset \R^n$ be an open set which is bounded in one direction and $0<s<\min(1,n/2)$. Assume that $\gamma_1,\gamma_2\in L^{\infty}(\R^n)$ with background deviations $m_1,m_2$ satisfy $\gamma_1(x),\gamma_2(x)\geq \gamma_0>0$ and $m_1,m_2\in H^{2s,\frac{n}{2s}}(\R^n)$. Moreover, assume that $m_0\vcentcolon =m_1-m_2\in H^s(\R^n)$. Finally, assume that $W_1,W_2 \subset \Omega_e$ are nonempty open sets and that  $\gamma_1|_{W_1\cup W_2} = \gamma_2|_{W_1\cup W_2}$ holds. Then
\begin{enumerate}[(i)]
    \item\label{item 1 UCP for DN map} if $W_1\cap W_2\neq \emptyset$, we have  $\left.\Lambda_{\gamma_1}f\right|_{W_2}=\left.\Lambda_{\gamma_2}f\right|_{W_2}$ for all $f\in C_c^{\infty}(W_1)$ if and only if $\gamma_1=\gamma_2$ in $\R^n$.
    \item\label{item 2 characterization of uniqueness} if $W_1\cap W_2=\emptyset$, we have $\left.\Lambda_{\gamma_1}f\right|_{W_2}=\left.\Lambda_{\gamma_2}f\right|_{W_2}$ for all $f\in C_c^{\infty}(W_1)$ if and only if $m_1-m_2$ is the unique solution of
    \begin{equation}
    \label{eq: PDE uniqueness cond eq}
        \begin{split}
            (-\Delta)^sm-\frac{(-\Delta)^sm_1}{\gamma_1^{1/2}}m&=0\quad\text{in}\quad \Omega,\\
            m&=m_0\quad\text{in}\quad \Omega_e.
        \end{split}
    \end{equation} 
\end{enumerate}
\end{theorem}

\begin{theorem}[Exterior determination]
\label{thm: exterior determination}
    Let $\Omega\subset \R^n$ be an open set which is bounded in one direction and $0<s<\min(1,n/2)$. Assume that $\gamma_1,\gamma_2\in L^{\infty}(\R^n)$ with background deviations $m_1,m_2$ satisfy $\gamma_1(x),\gamma_2(x)\geq \gamma_0>0$ and $m_1,m_2\in H^{2s,\frac{n}{2s}}(\R^n)$. 
    Suppose that $W \subset \Omega_e$ is a nonempty open set such that $\gamma_1,\gamma_2\in C(W)$.
If $\left.\Lambda_{\gamma_1}f\right|_{W}=\left.\Lambda_{\gamma_2}f\right|_{W}$ for all $f\in C_c^{\infty}(W)$, then $\gamma_1=\gamma_2$ in $W$.
\end{theorem}

The proofs of Theorems \ref{thm: ucp of dn map} and \ref{thm: exterior determination} are given in Sections \ref{sec: UCP DN-maps and invariance of data} and \ref{subsec: exterior determination}, respectively. It is evident that Theorem~\ref{thm: ucp of dn map} \ref{item 1 UCP for DN map} with $W_1 = W_2 := W$ and Theorem~\ref{thm: exterior determination} directly imply the following global uniqueness result, which is the main result of our work:

\begin{theorem}[Global uniqueness]
\label{thm: Global uniqueness}
    Let $\Omega\subset \R^n$ be an open set which is bounded in one direction and $0<s<\min(1,n/2)$. Assume that $\gamma_1,\gamma_2\in L^{\infty}(\R^n)$ with background deviations $m_1,m_2$ satisfy $\gamma_1(x),\gamma_2(x)\geq \gamma_0>0$ and $m_1,m_2\in H^{2s,\frac{n}{2s}}(\R^n) \cap H^{s}(\R^n)$. Suppose that $W \subset \Omega_e$ is a nonempty open set such that $\gamma_1,\gamma_2\in C(W)$.
Then $\left.\Lambda_{\gamma_1}f\right|_{W}=\left.\Lambda_{\gamma_2}f\right|_{W}$ for all $f\in C_c^{\infty}(W)$ if and only if $\gamma_1=\gamma_2$ in $\R^n$.
\end{theorem}

The case $W := W_1 \cap W_2 \neq \emptyset$ is essentially the most general one for which the global uniqueness may hold for general partial data problems with the knowledge of $\Lambda_{\gamma_1}f|_{W_2}=\Lambda_{\gamma_2}f|_{W_2}$ for all $f \in C_c^\infty(W_1)$. When $\Omega$ is a bounded domain, $W_1 \cap W_2 = \emptyset$ and $\overline{W_1\cup W_2} \subset \Omega_e$, it is always possible to find counterexamples to the global uniqueness (see~\cite[Theorem 1.2]{RZ2022FracCondCounter}). Similar counterexamples are also possible on domains that are bounded in one direction as proved in~\cite[Theorem 1.3]{RZ2022FracCondCounter} whenever $\dist(W_1\cup W_2,\Omega)>0$ and $0 < s < \min(1,n/4)$. There are still interesting open questions related to the existence of concrete counterexamples when these additional assumptions on $\Omega$, $W_1$ and $W_2$ do not hold, but we believe that this is merely due to some technical difficulties to verify the right regularity of possible counterexample candidates. Therefore, it is expected that the uniqueness is always lost in the case \ref{item 2 characterization of uniqueness} of Theorem~\ref{thm: ucp of dn map}.

Our main result is complemented by the following two, which are proved in Section~\ref{subsec: exterior determination}, and deal with the questions of stability and reconstruction in the exterior:

\begin{proposition}[Stability estimate]\label{prop:LipStability}
    Let $\Omega\subset \R^n$ be an open set which is bounded in one direction and $0<s<\min(1,n/2)$. Assume that $\gamma_1,\gamma_2\in L^{\infty}(\R^n)$ with background deviations $m_1,m_2$ satisfy $\gamma_1(x),\gamma_2(x)\geq \gamma_0>0$ and $m_{1},m_2\in H^{2s,\frac{n}{2s}}(\R^n)$. If $W \subset \Omega_e$ is a nonempty open set such that $\gamma_i\in C(W)$, $i=1,2$, then 
    \begin{equation}
    \label{eq: stability estimate}
        \|\gamma_1-\gamma_2\|_{L^{\infty}(W)}\leq 2^s\|\Lambda_{\gamma_1}-\Lambda_{\gamma_2}\|_{X\to X^*}.
    \end{equation}
\end{proposition}

\begin{proposition}[Exterior reconstruction formula]
\label{prop: exterior determination}
    Let $\Omega\subset \R^n$ be an open set which is bounded in one direction and $0<s<\min(1,n/2)$. Assume that $\gamma\in L^{\infty}(\R^n)$ with background deviations $m_{\gamma}$ satisfy $\gamma(x)\geq \gamma_0>0$ and $m_{\gamma}\in H^{2s,\frac{n}{2s}}(\R^n)$. If $W \subset \Omega_e$ is a nonempty, bounded, open set such that $\gamma\in C(W)$, $x_0\in W$, then there exist $(\phi_N)_{N \in \N}\subset C_c^{\infty}(W)$ and solutions $(u_N)_{N \in \N}\subset H^s(\R^n)$ of the homogeneous fractional conductivity equation  with exterior values $\phi_N$ such that
    \[
        \gamma(x_0)=\lim_{N\to\infty}E_{\gamma}(u_N).
    \]
\end{proposition}

We make some remarks about our results and assumptions next.

\begin{remark} The continuity assumption in Theorem \ref{thm: exterior determination} can be replaced by the assumption that there exists $\tilde{\gamma}_1,\tilde{\gamma}_2 \in L^\infty(\R^n)$ such that $\tilde{\gamma}_1,\tilde{\gamma}_2$ are continuous a.e. in $W$ and $\tilde{\gamma}_1=\gamma_1, \tilde{\gamma}_2=\gamma_2$ a.e. in $W$. In this case, the conclusion $\gamma_1 = \gamma_2$ holds a.e. in $W$. This follows from the proof of Theorem \ref{thm: exterior determination} with minor modifications. This observation carries over to Theorem \ref{thm: Global uniqueness}, and Propositions \ref{prop:LipStability} and \ref{prop: exterior determination}.
\end{remark}

\begin{remark}If $\delta > 0$ and one has that $m_\gamma \in H^{2s+\delta,n/2s}(\R^n) \cap H^s(\R^n)$, then the Sobolev embeddings into H\"older spaces give that $\gamma$ is continuous and Theorem~\ref{thm: Global uniqueness} is known to apply. Therefore, the continuity assumption in exterior determination is not very restrictive in the light of the regularity assumptions in Theorem~\ref{thm: ucp of dn map}. 
\end{remark}

\begin{remark}[Examples] Let $0 < c < 1$ be a small fixed paramater. Then the following conductivities satisfy the assumptions of Theorem~\ref{thm: Global uniqueness}:
\begin{enumerate}[(i)]
\item $\gamma=h+1$ for some $h \in C_c^\infty(\R^n)$ such that $h \geq c-1$.
\item $\gamma = (m+1)^2 = m^2 +2m +1$ for some $m \in H^{2s,n/2s}(\R^n) \cap H^{s}(\R^n) \cap L^\infty(\R^n)$ such that $m \geq c-1$ and $m \in C(W)$. We have that $m^2 +2m \in H^{2s,n/2s}(\R^n)$ by~\cite{AdamsComposition}.
\end{enumerate}
\end{remark}

\subsection{Motivation and connection to the literature}\label{sec: motivation}

In order to explain the motivations behind our study and the connections it bears with the known literature, in this section we compare the inverse problems for the classical and fractional conductivity equations (see also the surveys~\cite{Uh14, S17} for more information and references). This will also serve the purpose of showing the similarities in the mathematical structures of the two problems, the main differences being related to the geometric setting of the measurements and locality/nonlocality questions.

\subsubsection{Two conductivity equations} Let $\Omega\subset\R^n$ be a bounded open set, and consider a fixed but unknown conductivity function $\gamma:\R^n\rightarrow \R^+$. In the classical Calder\'on problem, which was first introduced in~\cite{C80}, the goal is to recover $\gamma$ in $\Omega$ from measurements of electric potentials and currents \emph{performed on the boundary $\partial\Omega$}. The electric potential $u$ solves the classical conductivity equation
\begin{equation}
\label{local-conductivity}
\begin{split}
    \Div(\gamma\nabla u)&=0 \quad \mbox{ in } \Omega, \\ u&=f \quad \mbox{ on } \partial\Omega
\end{split}
\end{equation}
    for any given boundary value $f$ on $\partial\Omega$, and the measurements are given in the form of the DN map $\Lambda_\gamma: H^{1/2}(\partial\Omega)\rightarrow H^{-1/2}(\partial\Omega)$ mapping $f \mapsto \gamma\partial_\nu u|_{\partial\Omega}$.

Let now $s\in(0,1)$. Using the nonlocal vector calculus developed in \cite{NonlocDiffusion}, it is possible to define a fractional conductivity equation \cite{C20, RZ2022unboundedFracCald} as
\begin{equation}\begin{split}\label{nonlocal-conductivity}
    \mbox{div}_s (\Theta_\gamma\cdot\nabla^su)&=0 \quad \mbox{ in } \Omega, \\ u&=f \quad \mbox{ in } \Omega_e ,
\end{split}\end{equation}
where $\Theta_\gamma$ is an appropriate matrix depending on the conductivity $\gamma$, and  $\nabla^s$, div$_s$ are the so called fractional gradient and divergence (see Section~\ref{subsec: Bessel potentials and co} for the definitions of such operators). The inverse problem for the fractional conductivity equation consists in recovering the conductivity $\gamma$ from a nonlocal analogue of DN data, which in the case of Lipschitz domains takes the form  $\Lambda^s_\gamma \colon H^s(\Omega_e)\rightarrow H^{-s}_{\overline \Omega_e}(\R^n)$.

Observe that while the conductivity operator $\nabla\cdot(\gamma\nabla)$ is local, the fractional conductivity operator $\Div_s(\Theta_\gamma\cdot\nabla^s)$ is nonlocal. As the former conserves supports, the values of $u$ in $\Omega_e$ do not interfere in the computation of $\nabla\cdot(\gamma\nabla u)$ in $\Omega$. Thus the classical Dirichlet problem~\eqref{local-conductivity} requires just a \emph{boundary value}. On the other hand, this does not suffice for the latter, given that the values of $u$ in $\Omega_e$ do interfere in the computation of div$_s(\Theta_\gamma\cdot\nabla^s u)$ in $\Omega$. This is why the fractional Dirichlet problem~\eqref{nonlocal-conductivity} needs an \emph{exterior value}. These fundamental characteristics also motivate the definitions of the DN maps as shown above, which both naturally derive from the bilinear forms associated with the problems~\eqref{local-conductivity} and \eqref{nonlocal-conductivity}.

\subsubsection{The Liouville reductions} By means of a standard change of variables, the so called \emph{Liouville reduction}, the classical Calder\'on problem is reduced to an inverse problem for the (time-independent) Schr\"odinger equation
\begin{equation}\begin{split}\label{liou-reduction}
    -\Delta u + q_\gamma u&=0 \quad \mbox{ in } \Omega \\ u&=g \quad \mbox{ on } \partial\Omega ,
\end{split}\end{equation}
which asks to recover the potential $q_\gamma$ given the relative DN map $\Lambda_{q_\gamma}$. It is then possible to infer results about the classical Calder\'on problem by studying the inverse problem for \eqref{liou-reduction}.

Our previous results~\cite{C20,RZ2022unboundedFracCald} show that the described inverse problem for the fractional conductivity equation also allows a  \emph{fractional Liouville reduction}, which transforms it into an inverse problem for the fractional Schr\"odinger equation  
\begin{equation}\begin{split}\label{frac-liou-reduction}
    (-\Delta)^s u + q u&=0 \quad \mbox{ in } \Omega \\ u&=g \quad \mbox{ in } \Omega_e ,
\end{split}\end{equation}
where $(-\Delta)^s$ is the fractional Laplacian (see Section~\ref{subsec: Bessel potentials and co} for the definition). The reduced problem consists in recovering the potential $q$ from the nonlocal DN map $\Lambda^s_q$, whose definition is based on the bilinear form associated to the problem. By analogy to the local case, the inverse problem obtained after the reduction is itself called the fractional Calder\'on problem. It was introduced in the seminal paper~\cite{GSU20}, and as such predates the fractional conductivity equation.

\subsubsection{Unique continuation principles and interior uniqueness} Being a prototypical second order elliptic equation, the pioneering work of Carleman \cite{C39}, see also~\cite{A57, AKS62, KT08}, shows that  \eqref{liou-reduction} admits the following strong unique continuation property (SUCP):
\\

\emph{If $u$ vanishes of infinite order
at $x_0$, then $u \equiv 0$ in a neighborhood of $x_0$.}
\\

\noindent This is proved using Carleman estimates, which are a special class of weighted $L^2$ estimates. These also allow one to prove the existence of complex geometrical optics (CGO) solutions to \eqref{liou-reduction}, which were devised by Sylvester and Uhlmann (\cite{SU87, SU86}) in order to emulate Calder\'on's exponential solutions (\cite{C80}) at high frequencies. Their  construction, which was in origin for $C^2$ conductivities for $n\geq 3$, has been upgraded for $C^1$~\cite{HT13}, Lipschitz~\cite{CR16} and $W^{1,n}(\Omega)\cap L^\infty(\Omega)$, $n=3,4$ conductivities~\cite{H15}.  By using such CGO solutions as test functions in an integral identity of the form $\int_{\R^n}qu_1u_2 dx=0$ (the so called \emph{Alessandrini identity}), one can prove uniqueness for the classical Calder\'on problem. In the case $n=2$, the fundamental result by Astala and Päivärinta~\cite{AP06} proved uniqueness for an $L^\infty(\Omega)$ conductivity using complex analysis techniques. Later on, this result was improved to $L^p(\Omega)$ for $p>4/3$ in~\cite{BTW19}. Given that the equation does not carry information about $\gamma$ in $\Omega_e$, all these are \emph{interior uniqueness} results. \emph{Partial data} results are available for $n\geq 3$ in specific geometries~\cite{DSFKSU09,Is07,KSU07}, and in general for $n=2$~\cite{IUY10}. 

The situation is quite different in this regard for the fractional Calder\'on problem. It is in fact known that the fractional Laplacian enjoys the unique continuation property
\\

\emph{Let $r\in \R$ and $u \in H^r(\R^n)$. If $u=(-\Delta)^su=0$ in a nonempty open set, then $u\equiv 0$.}
\\

\noindent Such property of course immediately extends to local perturbations of the fractional Laplacian, such as for the fractional Schr\"odinger equation~\eqref{frac-liou-reduction}. It is a classical result when $u$ is assumed to be more regular, and was proved for $s\in(0,1)$ in~\cite{GSU20} using Carleman estimates from~\cite{Ru15} and the Caffarelli--Silvestre extension~\cite{CS07}. This proof was then extended to general $s\in\R^+\setminus \N$ in~\cite{CMR20}. The strength of the unique continuation property stated above for the fractional Calder\'on problem translates into stronger uniqueness results than in the classical case. Already in~\cite{GSU20} uniqueness was proved in all dimensions $n\geq 2$ and for partial data in the case of an $L^\infty$ potential. The proof of uniqueness in $\Omega$ was extended to global low regularity potentials in $L^{n/2s}(\R^n)$ and $H^{-s,n/s}(\R^n)$ in~\cite{RS17a}. Even a single measurement was shown to be sufficient for uniqueness in~\cite{GRSU20}. Many perturbed versions of the fractional Calder\'on problem have been considered, among which~\cite{CMRU20,RZ2022unboundedFracCald} solved the uniqueness question for general local perturbations and~\cite{C21} for quasilocal perturbations. Our Theorem~\ref{thm: ucp of dn map} adds to the cited results by proving the interior uniqueness in the inverse problem for the fractional conductivity equation in the case when $\Omega$ is open and bounded in one direction and $\gamma$ is nontrivial in the exterior.

The main technique used in all of these results is quite different from the classical one. Using the fractional unique continuation result one can prove the Runge approximation property~\cite{GSU20, DSV14}, which allows one to approximate any sufficiently regular function on $\Omega$ by solutions to \eqref{frac-liou-reduction}. These are then used as test functions in the Alessandrini identity, thus avoiding the need for CGO solutions altogether.  

\subsubsection{Exterior and boundary determination} The recovery of a smooth conductivity $\gamma$ up to infinite order on the boundary $\partial\Omega$ has been established by Kohn and Vogelius in~\cite{KohnVogelius} and as an application they obtained the global recovery of piecewise real analytic conductivities~\cite{KV85}. These results were later extended to continuous conductivities~\cite{SU88, Brown}. Our Theorem~\ref{thm: exterior determination} serves the purpose of positively answering the very natural question of whether in the fractional Calder\'on problem $\gamma$ can be uniquely recovered in $\Omega_e$. By analogy to the boundary determination question studied by Kohn and Vogelius, we call such problem \emph{exterior determination}. Theorem~\ref{thm: Global uniqueness} gives conditions under which the global uniqueness can be achieved.

\subsubsection{Stability and reconstruction questions}
In~\cite{A88}, Alessandrini proved that in dimension $n\geq 3$ there exists a logarithmic modulus of continuity $$\|\gamma_1-\gamma_2\|_{L^\infty(\Omega)} \leq C \left( |\log \|\Lambda_1-\Lambda_2\|_{\frac{1}{2},-\frac{1}{2}}|^{-\sigma}+ \|\Lambda_1-\Lambda_2\|_{\frac{1}{2},-\frac{1}{2}} \right)\;,\quad \sigma\in(0,1)$$ for smooth conductivities in the case of the classical Calder\'on problem. Such estimate was then showed to be optimal by Mandache~\cite{M01}. This implies that the classical Calder\'on problem is severely ill-posed. However, Nachman and Novikov were able to show a reconstruction result for the Schr\"odinger and conductivity equations~\cite{N88, No88} from the associated DN map using CGO techniques.

 The fractional Calder\'on problem was shown to be severely ill-posed as well~\cite{RS17, RS17a, RS19, RS17d}. In fact, since the following (optimal) estimate holds $$ \|q_1-q_2\|_{L^{n/2s}(\Omega)} \leq C |\log \|\Lambda_{q_1}-\Lambda_{q_2}\|_* |^{-\sigma}\;,\quad \sigma\in (0,1)\;, $$ 
the fractional Calder\'on problem enjoys only logarithmic stability. In Proposition~\ref{prop:LipStability}, we show that Lipschitz stability \emph{in the exterior} can be achieved for the the fractional conductivity equation. Interior reconstruction of a low regularity potential was shown to be possible even in the case of a single measurement in~\cite{GRSU20}, and the result extends to the fractional conductivity equation by means of the fractional Liouville reduction for conductivities which are known to be trivial in the \emph{whole} exterior~\cite{C20}. Our Lemma~\ref{lemma: exterior determination lemma} gives a novel \emph{exterior} reconstruction result, which is reminiscent of the discussed boundary determination by Nachman and Novikov for the classical Calder\'on problem. 

\subsection{Organization of the rest of the article} We outline the mathematical structures behind the proof of Theorem~\ref{thm: Global uniqueness} in Section~\ref{sec:outlineofProofs}. We recall preliminaries on Bessel potential spaces, the fractional Laplacians, and the fractional Schrödinger and fractional conductivity equations in Section~\ref{sec:preliminaries}. We prove Theorem~\ref{thm: ucp of dn map} in Section~\ref{sec: UCP DN-maps and invariance of data}. We prove elliptic energy estimates for the related exterior value problems, construct special exterior conditions, and prove the exterior determination results in Sections \ref{subsec: elliptic energy estimates}, \ref{subsec: exterior conditions} and \ref{subsec: exterior determination}, respectively. We give a detailed proof of an auxiliary lemma for the used exterior conditions in Appendix~\ref{sec: proof lemma 4.1}.

\subsection*{Acknowledgements} G.C. was supported by an Alexander-von-Humboldt postdoctoral fellowship. J.R. was supported by the Vilho, Yrjö and Kalle Väisälä Foundation of the Finnish Academy of Science and Letters.

\section{Outline of the proof of Theorem~\ref{thm: Global uniqueness}}\label{sec:outlineofProofs}
The proof of Theorem~\ref{thm: exterior determination} and the related tools are developed rather completely within this article. This is not the case for Theorem~\ref{thm: ucp of dn map}, which is fundamentally based on the works~\cite{CS07,C20,GSU20,Ru15,RS17,RZ2022unboundedFracCald}, multiplication results in Bessel potential spaces (see e.g.~\cite{AdamsComposition,BrezisComposition,RS17,RZ2022unboundedFracCald}) and different equivalent characterizations of fractional Laplacians (see e.g.~\cite{DINEPV-hitchhiker-sobolev,Kw15,SilvestreFracObstaclePhd}). Therefore, we have decided to include here the details of the mathematical structures behind the proof of Theorem~\ref{thm: Global uniqueness} for greater clarity. We outline next the proofs of Theorems \ref{thm: ucp of dn map} and \ref{thm: exterior determination} to achieve this goal.

\subsection*{Outline of the proof of Theorem~\ref{thm: ucp of dn map} \ref{item 1 UCP for DN map}}
\begin{enumerate}[(i)]
\item We first define the fractional Liouville reduction given in Lemmas \ref{eq: well-posedness results and DN maps} and \ref{lemma: Relation DN maps}, which transforms the exterior DN maps of the fractional conductivity equation to the exterior DN maps of the fractional Schrödinger equation with the global singular potentials $q_\gamma = -\frac{(-\Delta)^s (\gamma^{1/2}-1)}{\gamma^{1/2}}$. This transformation was first introduced in~\cite{C20} and further generalized for global low regularity conductivities in \cite{RZ2022unboundedFracCald}. One uses the formula
\begin{equation}\label{eq:fracLapDeltaDef}
(-\Delta)^su(x)=-\frac{C_{n,s}}{2}\int_{\R^n}\frac{u(x+y)+u(x-y)-2u(x)}{|y|^{n+2s}}\,dy,
\end{equation}
weak$^*$ approximation of conductivities, and multiplication results of Bessel potential spaces in the proof of the Liouville reduction.
\item It follows that if $\gamma_1|_{W_1 \cup W_2} = \gamma_2|_{W_1 \cup W_2}$ and $\Lambda_{\gamma_1}f|_{W_2} = \Lambda_{\gamma_2}f|_{W_2}$ for all $f \in C_c^\infty(W_1)$, then $\Lambda_{q_1}f|_{W_2} = \Lambda_{q_2}f|_{W_2}$ as argued in the proof of Theorem~\ref{thm: ucp of dn map}. This is based on a modification of the arguments for some simpler special cases considered in~\cite{C20,RZ2022unboundedFracCald}.
\item\label{item:interiordet} Now one may invoke the interior uniqueness result of the fractional Calderón problem for globally defined singular potentials \cite{RS17}: If $\Lambda_{q_1}f|_{W_2} = \Lambda_{q_2}f|_{W_2}$ for all $f \in C_c^\infty(W_1)$, then $q_1 = q_2$ in $\Omega$. This is based on the UCP of the fractional Laplacians and the Runge approximation argument in~\cite{GSU20}. 
\item \label{item:exteriordetSchrö}Next we may use the exterior determination argument of~\cite{RZ2022unboundedFracCald}, which states that if $\Lambda_{q_1}f|_{W_2} = \Lambda_{q_2}f|_{W_2}$ for all $f \in C_c^\infty(W_1)$ and $W=W_1\cap W_2 \neq \emptyset$, then $q_1=q_2$ in $W$. This uses the earlier interior determination step, which already guarantees that $q_1 = q_2$ in $\Omega$.
\item\label{item:UCPstep} We may then use the assumption that $\gamma_1|_W = \gamma_2|_W$ and the knowledge
\begin{equation}-\frac{(-\Delta)^s (\gamma_1^{1/2}-1)}{\gamma_1^{1/2}} = q_1 = q_2 = -\frac{(-\Delta)^s (\gamma_2^{1/2}-1)}{\gamma_2^{1/2}}\quad\text{in $W$}
\end{equation} 
together with the UCP of the fractional Laplacians to conclude that $\gamma_1 = \gamma_2$ in $\R^n$ as desired.
\item The extension to domains bounded in one direction is based on the theory developed in~\cite{RZ2022unboundedFracCald}. The used low regularity assumptions are found by considering multiplication results of Bessel potentials spaces and basic mapping properties of the fractional Laplacians.
\end{enumerate}

\begin{remark} In the step \ref{item:interiordet}, it is important to use data $\ip{\Lambda_{q_1}f}{g} = \ip{\Lambda_{q_2}f}{g}$ such that $\supp(f)\cap\supp(g) = \emptyset$. In the step \ref{item:exteriordetSchrö} the data must instead verify $\supp(f)\cap\supp(g)\neq \emptyset$. Because of the steps \ref{item:exteriordetSchrö} and \ref{item:UCPstep}, we expect single or finite measurement uniqueness results to be impossible for the global inverse fractional conductivity problem. The proof of Theorem~\ref{thm: exterior determination} also gives support to this conjecture.
\end{remark}

Theorem~\ref{thm: ucp of dn map}~\ref{item 2 characterization of uniqueness} uses the same arguments up to the point \ref{item:interiordet}. The argument after that is slightly different and mainly of interest for counterexamples which were already discussed in the introduction. Concrete counterexamples on general domains were constructed in~\cite{RZ2022FracCondCounter}.

\subsection*{Outline of the proof of Theorem~\ref{thm: exterior determination}}
\begin{enumerate}[(i)]
\item \label{item:exteriorDet1}Define the Dirichlet energy first as
\begin{equation}E_{\gamma}(u)\vcentcolon = B_{\gamma}(u,u)= \int_{\R^{2n}}\Theta_{\gamma}\nabla^su\cdot\nabla^su\,dxdy.
\end{equation}
Notice then that it holds that $E_\gamma(u_f) = \ip{\Lambda_\gamma f}{f}_{X^*\times X}$ where $u_f$ is the unique solution of the fractional conductivity equation with the exterior condition $f$.
\item We first show an elliptic energy estimate for the solutions of the fractional Schrödinger equations (see Lemma~\ref{lemma: regularity lemma}): Let $W \subset \Omega_e$. If $f\in C_c^{\infty}(W)$ and $u_f\in H^s(\R^n)$ is the unique solution of
    \begin{equation}
        \begin{split}
            ((-\Delta)^s+q)u&= 0\quad\text{in}\quad\Omega,\\
            u&= f\quad\text{in}\quad\Omega_e,
        \end{split}
    \end{equation}
    then
    \begin{equation}
        \|u_f-f\|_{H^s(\R^n)}\leq C\|f\|_{L^2(W)}
    \end{equation}
    for some $C(n,s,|W|,\Omega,\dist(W,\Omega))>0$. This uses the quadratic definition of the fractional Laplacian
    \begin{equation}
        \ip{(-\Delta)^sf}{\phi} =\frac{C_{n,s}}{2}\int_{\R^{2n}}\frac{(f(x)-f(y))(\phi(x)-\phi(y))}{|x-y|^{n+2s}}dxdy.
    \end{equation}
    \item \label{item:ellipticenergy}We then extend this elliptic energy estimate to the fractional conductivity equation~\eqref{eq:conductivity intro} in Corollary \ref{cor: regularity estimate cond eq} by means of the fractional Liouville reduction.
    \item We consider sequences of exterior conditions $\phi_N \in C_c^\infty(W)$ such that $\norm{\phi_N}_{L^2(W)} \to 0$ as $N\to\infty$ and $\norm{\phi_N}_{H^s(\R^n)}=1$ for all $N \in \N$. Let $u_N \in H^s(\R^n)$ be the unique solutions to the conductivity equation~\eqref{eq:conductivity intro} with $u_N|_{\Omega_e} = \phi_N$. The elliptic energy estimate of \ref{item:ellipticenergy} and the given properties of the exterior conditions already guarantee that the Dirichlet energies $E_\gamma(u_N)$ and $E_\gamma(\phi_N)$ are asymptotically equivalent as $N \to \infty$. These exterior conditions are given in Lemma~\ref{lemma: exterior conditions} and are similar to the boundary conditions considered by Kohn and Vogelius in~\cite{KohnVogelius}.
    \item Given any $x_0 \in W$, we furthermore show that we may always construct the sequence $\phi_N$ so that $E_\gamma(\phi_N) \to \gamma(x_0)$ as $N\to\infty$. This means that the Dirichlet energies of the solutions $u_N$ concentrate at point $x_0$ in the limit. These technical details are given in Lemma~\ref{lemma: exterior determination lemma}.
    \item Since the previous step holds true at any point in $W$, one obtains by \ref{item:exteriorDet1} a reconstruction method for $\gamma$ in $W$ (see Proposition~\ref{prop: exterior determination}). Furthermore, Lipschitz stability for exterior determination holds as a byproduct (see Proposition~\ref{prop:LipStability}).
\end{enumerate}

\section{Preliminaries}\label{sec:preliminaries}

\subsection{Bessel potential spaces, fractional Laplacians and fractional gradients}
\label{subsec: Bessel potentials and co}
 We denote the space of Schwartz functions by $\schwartz(\R^n)$ and its dual space, the space of tempered distributions, by $\tempered(\R^n)$. The Fourier transform of $u\in\schwartz(\R^n)$ is defined as
\[
    \fourier u(\xi)\vcentcolon = \hat u(\xi) \vcentcolon = \int_{\R^n} u(x)e^{-ix \cdot \xi} \,dx.
\]
 $\fourier$ acts as an isomorphism on $\schwartz(\R^n)$, and by duality on $\tempered(\R^n)$ as well. We denote the inverse of the Fourier transform by $\ifourier$. The Bessel potential of order $s \in \R$ is the Fourier multiplier $\vev{D}^s\colon \tempered(\R^n) \to \tempered(\R^n)$, that is
\begin{equation}\label{eq: Bessel pot}
    \vev{D}^s u \vcentcolon = \ifourier(\vev{\xi}^s\widehat{u}),
\end{equation} 
where $\vev{\xi}\vcentcolon = (1+|\xi|^2)^{1/2}$. If $s \in \R$ and $1 \leq p \leq \infty$, the Bessel potential space $H^{s,p}(\R^n)$ is
\begin{equation}
\label{eq: Bessel pot spaces}
    H^{s,p}(\R^n) \vcentcolon = \{ u \in \tempered(\R^n)\,;\, \vev{D}^su \in L^p(\R^n)\},
\end{equation}
 endowed with the norm $\norm{u}_{H^{s,p}(\R^n)} \vcentcolon = \norm{\vev{D}^su}_{L^p(\R^n)}$. Given an open set $\Omega\subset \R^n$, we define the following local Bessel potential space:
\begin{equation}\begin{split}\label{eq: local bessel pot spaces}
    \widetilde{H}^{s,p}(\Omega) &\vcentcolon = \mbox{closure of } C_c^\infty(\Omega) \mbox{ in } H^{s,p}(\R^n).
    \end{split}
\end{equation}
We see that $\widetilde{H}^{s,p}(\Omega)$ is a closed subspace of $H^{s,p}(\R^n)$. As customary, we omit the index $p$ from the above notations in the case $p=2$. We shall also make use of the equivalent \emph{Gagliardo--Slobodeckij norms} 
\begin{equation}\label{eq:Gagliardonorms}
     \norm{u}_{W^s(\R^n)} := \sqrt{\norm{u}_{L^2(\R^n)}^2+\norm{(-\Delta)^{s/2}u}_{L^2(\R^n)}^2}
 \end{equation} for $u \in H^s(\R^n)$. We remark that in some references the norms may be different by a constant appearing in the second term $\norm{(-\Delta)^{s/2}u}_{L^2(\R^n)}$.

If $u\in\tempered(\R^n)$ is a tempered distribution and $s\geq 0$, the fractional Laplacian of order $s$ of $u$ is the Fourier multiplier
\[
    (-\Delta)^su\vcentcolon = \ifourier(|\xi|^{2s}\widehat{u}),
\]
whenever the right hand side is well-defined. If $p\geq 1$ and $t\in\R$, the fractional Laplacian is a bounded linear operator $(-\Delta)^{s}\colon H^{t,p}(\R^n) \to H^{t-2s,p}(\R^n)$. In the special case $u\in\schwartz(\R^n)$ and $s\in(0,1)$, we have the identities (see e.g.~\cite[Section 3]{DINEPV-hitchhiker-sobolev})
\begin{equation}
\begin{split}
    (-\Delta)^su(x)&=C_{n,s}\,\text{p.v.}\int_{\R^n}\frac{u(x)-u(y)}{|x-y|^{n+2s}}\,dy\\
    &= -\frac{C_{n,s}}{2}\int_{\R^n}\frac{u(x+y)+u(x-y)-2u(x)}{|y|^{n+2s}}\,dy,
\end{split}
\end{equation}
where $
    C_{n,s}:=\frac{4^s \Gamma(n/2+s)}{\pi^{n/2}|\Gamma(-s)|}.$
    
 Let now $s\in(0,1)$. The fractional gradient of order $s$ is the bounded linear operator $\nabla^s\colon H^s(\R^n)\to L^2(\R^{2n};\R^n)$ given by (see ~\cite{C20,NonlocDiffusion,RZ2022unboundedFracCald} and~\cite[Propositions 3.4 and 3.6]{DINEPV-hitchhiker-sobolev})
    \[
        \nabla^su(x,y):=\sqrt{\frac{C_{n,s}}{2}}\frac{u(x)-u(y)}{|x-y|^{n/2+s+1}}(x-y),
    \]
    with
    \begin{equation}
    \label{eq: bound on fractional gradient}
        \|\nabla^su\|_{L^2(\R^{2n})}=\|(-\Delta)^{s/2}u\|_{L^2(\R^n)}\leq \|u\|_{H^s(\R^n)}.
    \end{equation}
    The fractional divergence of order $s$ is the bounded linear operator \[\Div_s\colon L^2(\R^{2n};\R^n)\to H^{-s}(\R^n)\] given by
    \[
        \langle \Div_su,v\rangle_{H^{-s}(\R^n)\times H^s(\R^n)}=\langle u,\nabla^sv\rangle_{L^2(\R^{2n})}
    \]
    for all $u\in L^2(\R^{2n};\R^n),v\in H^s(\R^n)$, that is, $\Div_s$ is by definition the adjoint of $\nabla^s$. One can show that (see ~\cite[Section 8]{RZ2022unboundedFracCald})
    \[
        \|\Div_s(u)\|_{H^{-s}(\R^n)}\leq \|u\|_{L^2(\R^{2n})}
    \]
    for all $u\in L^2(\R^{2n};\R^n)$, and also $(-\Delta)^su=\Div_s(\nabla^su)$ weakly for all $u\in H^s(\R^n)$ (see ~\cite[Lemma 2.1]{C20}). 
    
    One also has the following product rule for the fractional gradient
    \begin{equation}\label{fractional-product-rule}
        \begin{split}
            &\nabla^s(\phi\psi)(x,y)=\sqrt{\frac{C_{n,s}}{2}}\frac{x-y}{|x-y|^{n/2+s}}(\phi(x)\psi(x)-\phi(y)\psi(y))\\
            &=\sqrt{\frac{C_{n,s}}{2}}\frac{x-y}{|x-y|^{n/2+s}}(\phi(x)\psi(x)-\phi(x)\psi(y)+\phi(x)\psi(y)-\phi(y)\psi(y))\\
            &=\phi(x)\nabla^s\psi(x,y)+\psi(y)\nabla^s\phi(x,y)
        \end{split}
    \end{equation}
    for all $\phi,\psi\in H^s(\R^n)$ and a.e. $x,y\in\R^n$. 
    
\subsection{Domains bounded in one direction and fractional Poincar\'e inequalities}

\begin{definition}\label{def:bounded1dir}
Let $n,k\in \N$ with $0< k\leq n$, and assume $\omega\subset \R^k$ is a bounded open set. The set $\Omega_{\infty}:=\R^{n-k}\times \omega \subset\R^n$ is a \emph{cylindrical domain}.

Let $n\in\N$ and assume $\Omega\subset\R^n$ is an open set. If there exist a cylindrical domain $\Omega_{\infty} \subset \R^n$ and a rigid Euclidean motion $A$ such that $\Omega \subset A(\Omega_{\infty})$, we say that $\Omega$ is \emph{bounded in one direction}. 
\end{definition}

\begin{theorem}[{Poincaré inequality (cf.~\cite[Theorem 2.2]{RZ2022unboundedFracCald})}]\label{thm:PoincUnboundedDoms} Let $\Omega\subset\R^n$ be an open set that is bounded in one direction. Suppose that $2 \leq p < \infty$ and $0\leq s\leq t < \infty$, or $1 < p < 2$, $1 \leq t < \infty$ and $0 \leq s \leq t$. Then there exists $C(n,p,s,t,\Omega)>0$ such that
    \begin{equation}
    \label{eq: poincare on L1}
        \|(-\Delta)^{s/2}u\|_{L^p(\R^n)}\leq C\|(-\Delta)^{t/2}u\|_{L^p(\R^n)}
    \end{equation}
    for all $u\in \Tilde{H}^{t,p}(\Omega)$.
\end{theorem}

\subsection{Bilinear forms and DN maps for the fractional conductivity and Schr\"odinger equations}

\begin{lemma}[{Definition of bilinear forms and conductivity matrix (cf.~\cite[Lemma 8.3]{RZ2022unboundedFracCald})}]\label{lemma: bilinear forms conductivity eq}
    Let $\Omega\subset\R^n$ be an open set, $0<s<\min(1,n/2)$, $q\in L^{\frac{n}{2s}}(\R^n)$, $\gamma\in L^{\infty}(\R^n)$ and define the conductivity matrix associated to $\gamma$ by
    \begin{equation}
    \label{eq: conductivity matrix}
        \Theta_{\gamma}\colon \R^{2n}\to \R^{n\times n},\quad \Theta_{\gamma}(x,y)\vcentcolon =\gamma^{1/2}(x)\gamma^{1/2}(y)\mathbf{1}_{n\times n}
    \end{equation}
    for $x,y\in\R^n$. Then the maps $B_{\gamma}\colon H^s(\R^n)\times H^s(\R^n) \to \R$ and $B_q\colon H^s(\R^n)\times H^s(\R^n)\to \R$ defined by
    \begin{equation}
    \label{eq: conductivity bilinear form}
        B_{\gamma}(u,v)\vcentcolon =\int_{\R^{2n}}\Theta_{\gamma}\nabla^su\cdot\nabla^sv\,dxdy
    \end{equation}
    and 
    \begin{equation}
    \label{eq: Schroedinger bilinear form}
        \quad B_q(u,v)\vcentcolon =\int_{\R^n}(-\Delta)^{s/2}u\,(-\Delta)^{s/2}v\,dx+\int_{\R^n}quv\,dx
    \end{equation}
    are continuous bilinear forms. Moreover, we have $q\in M_0(H^s\to H^{-s})$ with \begin{equation}
    \label{eq: potential as Sobolev multiplier}
        \|q\|_{s,-s}\leq C(1+\|q\|_{L^{\frac{n}{2s}}(\R^n)})
    \end{equation}
    for some $C>0$.
\end{lemma}

\begin{definition}[Weak solutions]
    Let $\Omega\subset\R^n$ be an open set, $0<s<\min(1,n/2)$, $q\in L^{\frac{n}{2s}}(\R^n)$ and $\gamma\in L^{\infty}(\R^n)$ with conductivity matrix $\Theta\colon \R^{2n}\to \R^{n\times n}$. If $f\in H^s(\R^n)$ and $F\in (\Tilde{H}^s(\Omega))^*$, then we say that $u\in H^s(\R^n)$ is a weak solution to the fractional conductivity equation
     \[
     \begin{split}
        \Div_s(\Theta\cdot\nabla^s u)&= F\quad\text{in}\quad\Omega,\\
        u&= f\quad\text{in}\quad\Omega_e,
     \end{split}
     \]
    if there holds
    \[
        B_{\gamma}(u,\phi)=F(\phi)\quad\text{and}\quad u-f\in\Tilde{H}^s(\Omega)
    \]
    for all $\phi\in \Tilde{H}^s(\Omega)$. 
    
    Similarly, we say that $v\in H^s(\R^n)$ is a weak solution to the fractional Schr\"odinger equation
    \[
            \begin{split}
            ((-\Delta)^s+q)v&=F\quad\text{in}\quad\Omega,\\
            v&=f\quad\text{in}\quad\Omega_e,
        \end{split}
    \]
    if there holds
    \[
        B_q(v,\phi)=F(\phi)\quad\text{and}\quad v-f\in\Tilde{H}^s(\Omega)
    \]
    for all $\phi\in \Tilde{H}^s(\Omega)$.
\end{definition}

By standard arguments and the fractional Poincar\'e inequality on domains bounded in one direction (Theorem~\ref{thm:PoincUnboundedDoms}) one can show the following result:

\begin{lemma}[{Well-posedness and DN maps (cf.~\cite[Lemma 8.10]{RZ2022unboundedFracCald})}]
\label{eq: well-posedness results and DN maps}
Let $\Omega\subset \R^n$ be an open set which is bounded in one direction and $0<s<\min(1,n/2)$. Assume that $\gamma\in L^{\infty}(\R^n)$ with conductivity matrix $\Theta_{\gamma}$, background deviation $m_{\gamma}$ and electric potential $q_{\gamma}\vcentcolon = -\frac{(-\Delta)^sm_{\gamma}}{\gamma^{1/2}}$ satisfies $\gamma(x)\geq \gamma_0>0$ and $m_{\gamma}\in H^{2s,\frac{n}{2s}}(\R^n)$. Then the following assertions hold:
    \begin{enumerate}[(i)]
        \item\label{item 1 well-posedness cond eq} For all $f\in X\vcentcolon = H^s(\R^n)/\Tilde{H}^s(\Omega)$ there are unique weak solutions $u_f,v_f\in H^s(\R^n)$ to the fractional conductivity equation
        \begin{equation}
        \label{eq: sol cond eq}
        \begin{split}
            \Div_s(\Theta_{\gamma}\cdot\nabla^s u)&= 0\quad\text{in}\quad\Omega,\\
            u&= f\quad\text{in}\quad\Omega_e
        \end{split}
        \end{equation}
        and to the fractional Schr\"odinger equation
        \begin{equation}
        \label{eq: frac Schrod eq}
            \begin{split}
            ((-\Delta)^s+q_{\gamma})v&=0\quad\text{in}\quad\Omega,\\
            v&=f\quad\text{in}\quad\Omega_e.
        \end{split}
        \end{equation}
        \item\label{item 2 well-posedness cond eq} The exterior DN maps $\Lambda_{\gamma}\colon X\to X^*$, $\Lambda_{q_{\gamma}}\colon X\to X^*$ given by 
        \begin{equation}
        \label{eq: identity DN maps}
        \begin{split}
            \langle \Lambda_{\gamma}f,g\rangle \vcentcolon =B_{\gamma}(u_f,g),\quad \langle \Lambda_{q_\gamma}f,g\rangle\vcentcolon =B_{q_{\gamma}}(v_f,g),
        \end{split}
        \end{equation}
        are well-defined bounded linear maps. 
    \end{enumerate}
\end{lemma}
\begin{remark}
    If no ambiguities arise, we will drop later in this article the subscripts attached to the conductivity matrix, the background deviation and the electric potential.
\end{remark}

\section{Unique continuation property of exterior DN maps and invariance of data}\label{sec: UCP DN-maps and invariance of data}

We give the proof of Theorem~\ref{thm: ucp of dn map} in this section. We begin with a simple lemma related to the fractional Liouville reduction.

\begin{lemma}
\label{lemma: Relation DN maps}
    Let $\Omega\subset\R^n$ be an open set which is bounded in one direction, $W\subset\Omega_e$ an open set and $0<s<\min(1,n/2)$. Assume that $\gamma,\Gamma\in L^{\infty}(\R^n)$ with background deviations $m_{\gamma},m_{\Gamma}$ satisfy $\gamma(x),\Gamma(x)\geq \gamma_0>0$ and $m_{\gamma},m_{\Gamma}\in H^{2s,\frac{n}{2s}}(\R^n)$. If $\gamma|_{W}=\Gamma|_{W}$, then
    \[
        \langle \Lambda_{\gamma}f,g\rangle=\langle \Lambda_{q_\gamma}(\Gamma^{1/2}f),(\Gamma^{1/2}g)\rangle
    \]
    holds for all $f,g\in\Tilde{H}^{s}(W)$.
\end{lemma}

\begin{proof}
    Observe that any $f\in \Tilde{H}^s(W)$ corresponds to a unique element of $X$. Hence the regularity assumptions on $\Gamma$ and~\cite[Corollary A.8]{RZ2022unboundedFracCald} imply that $\Gamma^{1/2}f\in\Tilde{H}^{s}(W)$ for any $f\in \Tilde{H}^s(W)$, which makes the DN maps in \eqref{eq: identity DN maps} well-defined. If $u_f\in H^s(\R^n)$ is the unique solution to 
    \[
    \begin{split}
    \Div_s(\Theta_\gamma\cdot\nabla^s u)&= 0\quad\text{in}\quad\Omega,\\
            u&= f\quad\text{in}\quad\Omega_e
    \end{split}
    \]
    with $f\in \Tilde{H}^s(W)$, by the fractional Liouville reduction (see~\cite[Theorem 8.6]{RZ2022unboundedFracCald}) $\gamma^{1/2}u_f\in H^s(\R^n)$ solves
    \[
        \begin{split}
            ((-\Delta)^s+q_\gamma)v&=0\quad\text{in}\quad\Omega,\\
            v&=\gamma^{1/2}f\quad\text{in}\quad\Omega_e.
        \end{split} 
    \] 
    Since $\gamma|_W = \Gamma|_W$, we have that $\gamma^{1/2}u_f$ is the unique solution to
    \[
            \begin{split}
            ((-\Delta)^s+q_\gamma)v&=0\quad\text{in}\quad\Omega,\\
            v&=\Gamma^{1/2}f\quad\text{in}\quad\Omega_e,
        \end{split}
    \] 
    which we denote by $v_{\Gamma^{1/2}f}$. 
    
    Therefore, by~\cite[Remark 8.8]{RZ2022unboundedFracCald}, we get
    \[
    \begin{split}
        \langle \Lambda_{\gamma}f,g\rangle&=B_{\gamma}(u_f,g)=B_{q_\gamma}(\gamma^{1/2}u_f,\gamma^{1/2}g)=B_{q_\gamma}(v_{\Gamma^{1/2}f},\Gamma^{1/2}g)\\
        &=\langle \Lambda_{q_\gamma}(\Gamma^{1/2}f),(\Gamma^{1/2}g)\rangle
    \end{split}
    \]
    for all $f,g\in\Tilde{H}^s(W)$.
\end{proof}

We can now prove Theorem~\ref{thm: ucp of dn map}, which generalizes~\cite[ Lemma~8.15]{RZ2022unboundedFracCald}. Here we focus mainly on the differences in the arguments which allow the generalization, and refer to the cited result for the missing details. 

\begin{proof}[{Proof of Theorem~\ref{thm: ucp of dn map}}]
    Throughout the proof let $\Gamma\in L^{\infty}(\R^n)$ be any function satisfying $\Gamma\geq \gamma_0$, $m_{\Gamma}\vcentcolon =\Gamma^{1/2}-1\in H^{2s,\frac{n}{2s}}(\R^n)$ and $\Gamma=\gamma_1=\gamma_2$ in $W_1\cup W_2$. Notice that there holds
    \[
        \langle \Lambda_{\gamma_1}f,g\rangle=\langle \Lambda_{\gamma_2}f,g\rangle
    \]
    for all $f\in \Tilde{H}^s(W_1)$, $g\in\Tilde{H}^{s}(W_2)$ if and only if we have
    \[
        \langle \Lambda_{\gamma_1}f,g\rangle=\langle \Lambda_{\gamma_2}f,g\rangle
    \]
    for all $f\in C_c^{\infty}(W_1)$, $g\in C_c^{\infty}(W_2)$. This follows by approximation and the continuity of $\Lambda_{\gamma_j}: X \to X^*$ for $j=1,2$ (see Lemma~\ref{eq: well-posedness results and DN maps} \ref{item 2 well-posedness cond eq}). Similarly, one can show that there holds
     \[
    \langle \Lambda_{q_1}f,g\rangle=\langle \Lambda_{q_2}f,g\rangle
    \]
    for all $f\in C_c^{\infty}(W_1)$, $g\in C_c^{\infty}(W_2)$ if and only if
    \[
    \langle \Lambda_{q_1}f,g\rangle=\langle \Lambda_{q_2}f,g\rangle
    \]
for all $f\in \Tilde{H}^s(W_1)$, $g\in\Tilde{H}^{s}(W_2)$.

Now, by the first part of the proof and Lemma~\ref{lemma: Relation DN maps}, the condition $\Lambda_{\gamma_1}f|_{W_2}=\Lambda_{\gamma_2}f|_{W_2}$ for all $f\in C_c^{\infty}(W_1)$ is equivalent to 
\begin{equation}
\label{eq: equivalence 1}
     \langle \Lambda_{q_1}(\Gamma^{1/2}f),(\Gamma^{1/2}g)\rangle=\langle \Lambda_{q_2}(\Gamma^{1/2}f),(\Gamma^{1/2}g)\rangle
\end{equation}
for all $f\in \Tilde{H}^s(W_1)$, $g\in \Tilde{H}^s(W_2)$. Next note that we can write 
\[
    \frac{1}{\Gamma^{1/2}}=1-\frac{m_{\Gamma}}{m_{\Gamma}+1}
\]
and set $\Gamma_0\vcentcolon =\min(0,\gamma_0^{1/2}-1)$. Let $H\in C^2_b(\R)$ satisfy $H(t)=\frac{t}{t+1}$ for $t\geq \Gamma_0$. By~\cite[p.~156]{AdamsComposition} and $m_{\Gamma}\in H^{2s,\frac{n}{2s}}(\R^n)\cap L^{\infty}(\R^n)$, we deduce $H(m_{\Gamma})\in H^{2s,\frac{n}{2s}}(\R^n)$, but since $m_{\Gamma}\geq \gamma_0^{1/2}-1$ it follows that $\frac{m_{\Gamma}}{m_{\Gamma}+1}\in H^{2s,\frac{n}{2s}}(\R^n)\cap L^{\infty}(\R^n)$. Therefore, by~\cite[Corollary A.8]{RZ2022unboundedFracCald} we see that \eqref{eq: equivalence 1} is equivalent to
\begin{equation}
\label{eq: equivalence 2}
     \langle \Lambda_{q_1}f,g\rangle=\langle \Lambda_{q_2}f,g\rangle  
\end{equation}
for all $f\in \Tilde{H}^s(W_1)$, $g\in \Tilde{H}^s(W_2)$. 

Now we can follow the proof of Lemma~8.15 in~\cite{RZ2022unboundedFracCald} to conclude that this is equivalent to the assertion that $m_1-m_2\in H^s(\R^n)$ is the unique solution of
    \begin{equation}
    \label{eq: sharpness}
        \begin{split}
            (-\Delta)^sm-\frac{(-\Delta)^sm_1}{\gamma_1^{1/2}}m&=0\quad\text{in}\quad \Omega,\\
            m&=m_0\quad\text{in}\quad \Omega_e,
        \end{split}
    \end{equation}
 and 
 \begin{equation}
 \label{eq: UCP integral}
    \int_{\Omega_e}\frac{(-\Delta)^sm_0}{\Gamma^{1/2}}fg\,dx=0
 \end{equation}
 for all $f\in C^{\infty}_c(W_1)$, $g\in C_c^{\infty}(W_2)$. 
 
 If the sets $W_1,W_2\subset\Omega_e$ are disjoint, we immediately obtain the statement \ref{item 2 characterization of uniqueness}. Assume now that $W_1\cap W_2\neq \emptyset$, and let $\Gamma^{1/2}_{\epsilon}\vcentcolon = \Gamma^{1/2}\ast \rho_{\epsilon}$, where $\rho_{\epsilon}$ is the standard mollifier. By standard arguments, one deduces $\Gamma_{\epsilon}^{1/2}\in C^{\infty}_b(\R^n)$ and $\Gamma_{\epsilon}^{1/2}\overset{\ast}{\rightharpoonup} \Gamma^{1/2}$. Therefore, by taking $g\vcentcolon = \Gamma_{\epsilon}^{1/2}h\in C^{\infty}_c(W_2)$ with $h\in C_c^{\infty}(W_2)$ in \eqref{eq: UCP integral} we see that there holds
 \[
    \int_{\Omega_e}\left(\frac{(-\Delta)^sm_0}{\Gamma^{1/2}}fh\right)\Gamma_{\epsilon}^{1/2}\,dx=0
 \]
 for all $f\in C_c^{\infty}(W_1)$, $h\in C_c^{\infty}(W_1)$. Since the expression in the brackets belongs to $L^1(\R^n)$ and $\Gamma_{\epsilon}^{1/2}\overset{\ast}{\rightharpoonup} \Gamma^{1/2}$ in $L^{\infty}(\R^n)$, we obtain in the limit $\epsilon\to 0$ the identity
 \[
    \int_{\Omega_e}(-\Delta)^sm_0fh\,dx=0
 \]
 for all $f\in C_c^{\infty}(W_1)$, $h\in C_c^{\infty}(W_1)$. 
 
 As $W_1\cap W_2\neq\emptyset$, we can choose an open bounded set $\omega\subset\R^n$ such that $\overline{\omega}\subset W_1\cap W_2$ and a cutoff function $f\in C_c^{\infty}(W_1)$ such that $f|_{\overline{\omega}}=1$. This gives
    \[
        \int_{\omega}(-\Delta)^sm_0g\,dx=0
    \]
    for all $g\in C_c^{\infty}(\omega)$, which in turn implies $(-\Delta)^sm_0=0$ in $\omega$. On the other hand, we have by the assumption $\gamma_1=\gamma_2$ in $\omega$  and hence the UCP (cf.~\cite[Theorem 1.2]{GSU20}) implies $m_0\equiv 0$. This shows $\gamma_1=\gamma_2$, which concludes the proof of \ref{item 1 UCP for DN map}.
\end{proof}

\begin{remark}
    Note that by definition of the dual map and~\cite[Corollary A.11]{RZ2022unboundedFracCald} the second assertion is equivalent to
 \[
    \langle ((-\Delta)^sm_0)f,\Gamma^{-1/2}g\rangle_{H^{-s}(\R^n)\times H^s(\R^n)}=0
 \]
 for all $f\in C^{\infty}_c(W_1)$, $g\in C_c^{\infty}(W_2)$. Now using $\Gamma^{-1/2}=1-\frac{m_{\Gamma}}{m_{\Gamma+1}}$ with $\frac{m_{\Gamma}}{m_{\Gamma+1}}\in H^{2s,\frac{n}{2s}}(\R^n)\cap L^{\infty}(\R^n)$,~\cite[Corollary A.8]{RZ2022unboundedFracCald} and~\cite[Corollary A.11]{RZ2022unboundedFracCald} we see that this is equivalent to
 \[
     \langle ((-\Delta)^sm_0)f,\Gamma^{-1/2}g\rangle_{H^{-s}(\R^n)\times H^s(\R^n)}=0
\]
$f\in \Tilde{H}^s(W_1)$, $g\in\Tilde{H}^{s}(W_2)$. Applying once more~\cite[Corollary A.8]{RZ2022unboundedFracCald} we deduce that this is the same as
 \[
     \langle ((-\Delta)^sm_0)f,g\rangle_{H^{-s}(\R^n)\times H^s(\R^n)}=0
\]
for all $f\in \Tilde{H}^s(W_1)$, $g\in\Tilde{H}^{s}(W_2)$. This directly shows that $(-\Delta)^sm_0$ vanishes as a distribution on $ W_1\cap W_2$. 
\end{remark}
\begin{remark}
    Suppose that $1/4 < s < 1$. Observe that if the assumptions $\gamma_1=\gamma_2$ in some measurable set $A\subset W_1\cap W_2$ with positive measure, $m_0\vcentcolon = m_1-m_2\in H^r(\R^n)$ for some $r\in\R$ and $\Lambda_{\gamma_1}f|_{W_2}=\Lambda_{\gamma_2}f|_{W_2}$ for all $f\in C_c^{\infty}(W_1)$ still implied that $(-\Delta)^sm_0=0$ in some open set $V\subset W_1\cap W_2$ containing $A$, then using~\cite[Theorem 3, Remark 5.6]{GRSU20} we could conclude that there holds $\gamma_1=\gamma_2$ in $\R^n$. The main obstruction in carrying out this analysis is that in this setting we only have the identity
    \[
        \langle\Lambda_{q_1}(\gamma_1^{1/2}f), (\gamma^{1/2}_1g)\rangle =\langle\Lambda_{q_2}(\gamma_2^{1/2}f), (\gamma^{1/2}_2g)\rangle
    \]
    for all $f\in\Tilde{H}^s(W_1)$, $g\in\Tilde{H}^s(W_2)$.
\end{remark}

\section{Exterior determination for the fractional conductivity equation}

We prove in this section our exterior determination results. We begin by proving $H^s \lesssim L^2$ energy estimates for the related exterior value problems in Section~\ref{subsec: elliptic energy estimates}. We then construct sequences of exterior conditions such that the fractional energies can be localized in the limit to any fixed point in $\R^n$. Finally, we combine these two important steps to obtain the exterior determination results in the end of this section.

\subsection{Elliptic energy estimates for the fractional Schrödinger and conductivity equations}\label{subsec: elliptic energy estimates}

\begin{definition}
    Let $s\in\R$. We say that a bilinear form $q\colon H^s(\R^n)\times H^s(\R^n)\to \R$ is local if $q(u,v)=0$ whenever $u,v$ have disjoint supports.
\end{definition}

\begin{definition}
    Let $\Omega\subset \R^n$ be an open set such that $\Omega_e\neq \emptyset$ and $s\in\R_+$. For any local bounded bilinear form $q\colon H^s(\R^n) \times H^s(\R^n) \to \R$, we define the $q-$perturbed bilinear form $B_q\colon H^s(\R^n)\times H^s(\R^n)\to \R$ by
     \begin{equation}
     \label{eq: bilinear form}
        B_q(u,v)\vcentcolon = \langle (-\Delta)^{s/2}u,(-\Delta)^{s/2}v\rangle +q(u,v)
     \end{equation}
     for all $u,v\in H^s(\R^n)$. We say that $u\in H^s(\R^n)$ is a weak solution to 
     \begin{equation}
     \label{eq: abstract PDE}
         \begin{split}
            ((-\Delta)^s+q)u&= F\quad\text{in}\quad\Omega,\\
            u&= f\quad\text{in}\quad\Omega_e
        \end{split}
     \end{equation}
     with $f\in H^s(\R^n), F\in(\widetilde{H}^s(\Omega))^*$, if $u-f\in\widetilde{H}^s(\Omega)$ and there holds
     \[
        B_q(u,\phi)=F(\phi)
     \]
     for all $\phi\in\widetilde{H}^s(\Omega)$.
\end{definition}

\begin{lemma}[Energy estimate for abstract Schr\"odinger equations]
\label{lemma: regularity lemma}
     Let $0<s<1$. Assume that $\Omega, W \subset \R^n$ are nonempty open sets such that $|W|<\infty$ and $W\Subset\Omega_e$. Let $q\colon H^s(\R^n) \times H^s(\R^n) \to \R$ be a local bounded bilinear form and let $B_q$ be the $q-$perturbed bilinear form. Suppose that for any $F\in (\widetilde{H}^s(\Omega))^*$ there is a unique solution $u_F\in H^s(\R^n)$ of
     \begin{equation}
     \label{eq: homogeneous abstract PDE}
         \begin{split}
            ((-\Delta)^s+q)u&= F\quad\text{in}\quad\Omega,\\
            u&= 0\quad\text{in}\quad\Omega_e
        \end{split}
     \end{equation}
     and $u_F$ satisfies the estimate
     \begin{equation}
     \label{eq: continuous dependence on data}
         \|u_F\|_{H^s(\R^n)}\leq C\|F\|_{(\widetilde{H}^s(\Omega))^*}
     \end{equation}
     for some $C>0$ independent of $F$.
    Then for any $f\in C_c^{\infty}(W)$ there is a unique solution $u_f\in H^s(\R^n)$ of
    \begin{equation}
    \label{eq: weak sol inhom}
        \begin{split}
            ((-\Delta)^s+q)u&= 0\quad\text{in}\quad\Omega,\\
            u&= f\quad\text{in}\quad\Omega_e
        \end{split}
    \end{equation}
    and there holds
    \begin{equation}
    \label{eq: regularity estimate}
        \|u_f-f\|_{H^s(\R^n)}\leq C\|f\|_{L^2(W)}
    \end{equation}
    for some constant $C>0$ depending only on $n,s,|W|$ and \emph{dist}$(W,\Omega)$.
\end{lemma}

\begin{proof}
    Let $f\in C_c^{\infty}(W)$ and $F\vcentcolon =-B_q(f,\cdot)$. Since $B_q$ is bounded, we know that $F\in (\widetilde{H}^s(\Omega))^*$, and thus the unique solution $u_F$ of the problem~\eqref{eq: homogeneous abstract PDE} is well-defined. By construction it also holds that $u_F+f$ solves the problem~\eqref{eq: weak sol inhom}. However, since the unique solution of \eqref{eq: homogeneous abstract PDE} with vanishing inhomogeneity must necessarily be trivial, solutions to \eqref{eq: weak sol inhom} are unique. Thus $u_F + f $ must be the \emph{unique} solution $u_f$ to \eqref{eq: weak sol inhom}, and by \eqref{eq: continuous dependence on data} 
    \[
        \|u_f-f\|_{H^s(\R^n)}\leq C\|B_q(f,\cdot)\|_{(\widetilde{H}^s(\Omega))^*}.
    \]
    
    Since $q$ is local and $C_c^{\infty}(\Omega)$ is dense in $\widetilde{H}^s(\Omega)$, we have
    \begin{equation}\label{estimate-n1}
    \begin{split}
        \|u_f-f\|_{H^s(\R^n)}&\leq C\sup_{\substack{\phi\in C_c^{\infty}(\Omega):\\ \|\phi\|_{H^s(\R^n)}=1}}|\langle (-\Delta)^{s/2}f,(-\Delta)^{s/2}\phi\rangle|.
    \end{split}
    \end{equation}
    Moreover, by~\cite[Theorem 1.1, (g)]{Kw15}
    $$ |\langle (-\Delta)^{s/2}f,(-\Delta)^{s/2}\phi\rangle| = \frac{C_{n,s}}{2}\left|\int_{\R^{2n}}\frac{(f(x)-f(y))(\phi(x)-\phi(y))}{|x-y|^{n+2s}}dxdy\right|. $$
    If $(x,y)\in (\Omega^c \times \Omega^c)\cup (W^c\times W^c)$, the integrand function vanishes by the support conditions of $f$ and $\phi$. On the other hand $W \subset \Omega^c$, and thus the integration set on the right hand side can be reduced to $(\Omega\times W) \cup (W\times \Omega)$. However, the integral over $\Omega\times W$ coincides to the one over $W\times \Omega$ by symmetry, and we are left with
    $$ |\langle (-\Delta)^{s/2}f,(-\Delta)^{s/2}\phi\rangle| = C_{n,s}\left|\int_{\Omega\times W}\frac{f(y)\phi(x)}{|x-y|^{n+2s}}\,dxdy\right|. $$
    
    By the Cauchy--Schwartz and Minkowski inequalities we have
    \[
        \begin{split}
            &\left|\int_{\Omega\times W}\frac{f(y)\phi(x)}{|x-y|^{n+2s}}\,dxdy\right|\\
            &\quad=\left|\int_{\Omega}\phi(x)\left(\int_{W}\frac{f(y)}{|x-y|^{n+2s}}\,dy\right)dx\right|\\
            &\quad\leq \|\phi\|_{L^2(\Omega)}\left\|\int_{W}\frac{f(y)}{|x-y|^{n+2s}}\,dy\right\|_{L^2(\Omega)}\\
            &\quad\leq \|\phi\|_{L^2(\Omega)}\left(\int_{W}\left(\int_{\Omega}\frac{|f(y)|^2}{|x-y|^{2n+4s}}\,dx\right)^{1/2}\,dy\right)\\
            &\quad=\|\phi\|_{L^2(\Omega)}\left(\int_{W}|f(y)|\left(\int_{\Omega}\frac{1}{|x-y|^{2n+4s}}\,dx\right)^{1/2}\,dy\right).
        \end{split}
    \]
    Observe that since $W\Subset\Omega_e$, for $(x,y)\in\Omega\times W$ it holds $|x-y|\geq \mbox{dist}(\Omega,W)=:r>0$, and thus
    \[
    \begin{split}
        \int_{\Omega}\frac{1}{|x-y|^{2n+4s}}\,dx&\leq \int_{B_r(y)^c}\frac{1}{|x-y|^{2n+4s}}\,dx=\frac{\omega_n}{(n+4s)r^{n+4s}}.
    \end{split}
    \]
   Hence, by H\"older's inequality we get
    \[
        \begin{split}
            \left|\int_{\Omega\times W}\frac{f(y)\phi(x)}{|x-y|^{n+2s}}\,dxdy\right|&\leq \sqrt{\frac{\omega_n}{(n+4s)r^{n+4s}}}|W|^{1/2}\|f\|_{L^2(W)}\|\phi\|_{L^2(\Omega)},
        \end{split}
    \]
    and eventually
    \[
        \|u_f-f\|_{H^s(\R^n)}\leq C C_{n,s} \sqrt{\frac{\omega_n}{(n+4s)r^{n+4s}}}|W|^{1/2}\|f\|_{L^2(W)}
    \]
    where $C > 0$ is the constant from the elliptic a priori estimate~\eqref{eq: continuous dependence on data}.
\end{proof}

\begin{corollary}[Energy estimate for conductivity equations]
\label{cor: regularity estimate cond eq}
    Let $0<s<\min(1,n/2)$. Suppose that $\Omega\subset\R^n$ is an open set which is bounded in one direction, $W\subset\Omega_e$ an open set with finite measure and positive distance from $\Omega$. Assume that $\gamma\in L^{\infty}(\R^n)$ with conductivity matrix $\Theta$ and background deviation $m$ satisfies $\gamma(x)\geq \gamma_0>0$ and $m\in H^{2s,\frac{n}{2s}}(\R^n)$. Then for any $f\in C_c^{\infty}(W)$ the associated unique solution $u_f\in H^s(\R^n)$ of
    \[
     \begin{split}
        \Div_s(\Theta\cdot\nabla^s u)&= 0\quad\text{in}\quad\Omega,\\
        u&= f\quad\text{in}\quad\Omega_e,
     \end{split}
     \]
     satisfies the estimate
     \[
        \|u_f-f\|_{H^s(\R^n)}\leq C\|f\|_{L^2(W)}
     \]
     for some $C>0$ depending only on $n,s,\gamma_0,\|\gamma\|_{L^{\infty}(\R^n)},|W|$, the distance between $W$ and $\Omega$, and the Poincar\'e constant of $\Omega$.
\end{corollary}

\begin{proof}
    First note that by~\cite[Lemma 8.10]{RZ2022unboundedFracCald} for any $f\in C_c^{\infty}(W)$ there is a unique solution $u_f\in H^s(\R^n)$ solving
    \[
     \begin{split}
        \Div_s(\Theta\cdot\nabla^s u)&= 0\quad\text{in}\quad\Omega,\\
        u&= f\quad\text{in}\quad\Omega_e
     \end{split}
    \]
    and since they are constructed via the Lax--Milgram theorem (cf.~\cite[Lemma 8.10]{RZ2022unboundedFracCald}) there holds
    \[
    \begin{split}
        \|u_f-f\|_{H^s(\R^n)}&\leq C\|B_{\gamma}(f,\cdot)\|_{(\widetilde{H}^s(\Omega))^*}
    \end{split}
    \]
    where $C(s,\Omega) > 0$ depends on the fractional Poincaré constant of $\Omega$.
    
    By~\cite[Remark 8.8]{RZ2022unboundedFracCald} and density of $C_c^{\infty}(\Omega)$ in $\widetilde{H}^s(\Omega)$ this gives
    \[\small{
    \begin{split}
         &\|u_f-f\|_{H^s(\R^n)}\leq C\sup_{\substack{\phi\in C_c^{\infty}(\Omega):\\ \|\phi\|_{H^s(\R^n)}=1}}|\langle \Theta\nabla^sf,\nabla^s\phi\rangle_{L^2(\R^{2n})}|\\
         &=C'\sup_{\substack{\phi\in C_c^{\infty}(\Omega):\\ \|\phi\|_{H^s(\R^n)}=1}}|\langle (-\Delta)^{s/2}(\gamma^{1/2}f),(-\Delta)^{s/2}(\gamma^{1/2}\phi)\rangle_{L^2(\R^{n})}+\langle q\gamma^{1/2}f,\gamma^{1/2}\phi\rangle_{L^2(\R^n)}|\\
         &=C'\sup_{\substack{\phi\in C_c^{\infty}(\Omega):\\ \|\phi\|_{H^s(\R^n)}=1}}|\langle (-\Delta)^{s/2}(\gamma^{1/2}f),(-\Delta)^{s/2}(\gamma^{1/2}\phi)\rangle_{L^2(\R^{n})}|,
    \end{split}}
    \]
    where $q\vcentcolon = -\frac{(-\Delta)^sm}{\gamma^{1/2}}\in L^{n/2s}(\R^n)$, and the last equality follows from the fact that $f$ and $\phi$ have disjoint support. 
    
    Let $(\rho_{\epsilon})_{\epsilon>0}\subset C_c^{\infty}(\R^n)$ be the standard mollifiers and define the sequences $\gamma_{\epsilon}^{1/2}\vcentcolon =\gamma^{1/2}\ast \rho_{\epsilon}\in C_b^{\infty}(\R^n)$ and $m_{\epsilon}\vcentcolon =m\ast \rho_{\epsilon}\in C_b^{\infty}(\R^n)$. Observe that $m_{\epsilon}=\gamma_{\epsilon}^{1/2}-1$, since $\int_{\R^n}\rho_{\epsilon}\,dx=1$. By the properties of mollification in $L^p$ spaces and the Bessel potential operator, we deduce
    \begin{enumerate}[(I)]
        \item $\gamma_{\epsilon}^{1/2}\overset{\ast}{\rightharpoonup} \gamma^{1/2}$ in $L^{\infty}(\R^n)$,\label{item:LiovilleProof1}
        \item $0<\gamma_0^{1/2}\leq \gamma_{\epsilon}^{1/2}\in L^{\infty}(\R^n)$ with $\|\gamma_{\epsilon}^{1/2}\|_{L^{\infty}(\R^n)}\leq \|\gamma\|_{L^{\infty}(\R^n)}^{1/2}$,\label{item:LiovilleProof2}
        \item $m_{\epsilon}\to m$ in $H^{2s,\frac{n}{2s}}(\R^n)$,\label{item:LiovilleProof3}
        \item $m_{\epsilon}\in L^{\infty}(\R^n)$ with $\|m_{\epsilon}\|_{L^{\infty}(\R^n)}\leq \|m\|_{L^{\infty}(\R^n)}\leq 1+\|\gamma\|_{L^{\infty}(\R^n)}^{1/2}$.\label{item:LiovilleProof4}
    \end{enumerate} 
    Using~\cite[Corollary A.7]{RZ2022unboundedFracCald}, we deduce
    \[
        m_{\epsilon}u\to mu\quad\text{in}\quad H^s(\R^n)
    \]
    as $\epsilon\to 0$ for any $u\in H^s(\R^n)$ and therefore by continuity of the fractional Laplacian we have
    \[
    \begin{split}
        &\langle (-\Delta)^{s/2}(\gamma_{\epsilon}^{1/2}f),(-\Delta)^{s/2}(\gamma_{\epsilon}^{1/2}\phi)\rangle_{L^2(\R^n)}\\
        &\quad=\langle (-\Delta)^{s/2}((m_{\epsilon}+1)f),(-\Delta)^{s/2}((m_{\epsilon}+1)\phi)\rangle_{L^2(\R^n)}\\
       &\quad\rightarrow \langle (-\Delta)^{s/2}(\gamma^{1/2}f),(-\Delta)^{s/2}(\gamma^{1/2}\phi)\rangle_{L^2(\R^n)}
    \end{split}
    \]
    as $\epsilon\to 0$. Now on the $L^2$ inner product on the left hand side we can apply the calculation of Lemma~\ref{lemma: regularity lemma} and deduce by using \ref{item:LiovilleProof2} the estimate
    \[\begin{split}
        &|\langle (-\Delta)^{s/2}(\gamma_{\epsilon}^{1/2}f),(-\Delta)^{s/2}(\gamma_{\epsilon}^{1/2}\phi)\rangle_{L^2(\R^n)}|\\
        &\quad\leq CC_{n,s}\|\gamma\|_{L^{\infty}(\R^n)}\sqrt{\frac{\omega_n}{(n+4s)r^{n+4s}}}|W|^{1/2}\|f\|_{L^2(W)}\|\phi\|_{L^2(\Omega)}.
        \end{split}
    \]
    This shows
    \[
        \|u_f-f\|_{H^s(\R^n)}\leq CC_{n,s}\|\gamma\|_{L^{\infty}(\R^n)}\sqrt{\frac{\omega_n}{(n+4s)r^{n+4s}}}|W|^{1/2}\|f\|_{L^2(W)}.\qedhere
    \]
\end{proof}

\subsection{Construction of exterior conditions with an energy concetration property}\label{subsec: exterior conditions}

For any $0<s<\min(1,n/2)$ and positive conductivity $\gamma\in L^{\infty}(\R^n)$ with conductivity matrix $\Theta$ we denote the associated energy by 
\[
    E_{\gamma}(u)\vcentcolon = B_{\gamma}(u,u)= \int_{\R^{2n}}\Theta\nabla^su\cdot\nabla^su\,dxdy
\]
for all $u\in H^s(\R^n)$.

\begin{lemma}[Exterior conditions]
\label{lemma: exterior conditions}
    Let $\Omega \subset \R^n$ be open with nonempty exterior, $0<s<\min(1,n/2)$, $x_0\in W \subset \Omega_e$ for an open set $W$, and $M\in\N$. There exists a sequence $(\phi_N)_{N\in\N}\subset C_c^{\infty}(W)$ such that
    \begin{enumerate}[(i)]
        \item\label{item 1: normalization} for all $N\in\N$ it holds $\|\phi_N\|_{H^s(\R^n)}=1$, 
        \item \label{item 2: estimate} for all $t\geq -M$ there exist constants $C_{t,s},C'_{t,s}>0$ such that 
        $$C'_{t,s}N^t\leq \|\phi_N\|_{H^{t+s}(\R^n)}\leq C_{t,s}N^{t} \quad \mbox{for all } N\in\N,
        $$
        \item \label{item 3: support} $\supp(\phi_N)\to \{x_0\}$
        as $N\to\infty$. 
    \end{enumerate}
\end{lemma}

Kohn and Vogelius proved a similar result in their celebrated work on boundary determination for the conductivity equation (cf.~\cite[Lemma 1]{KohnVogelius}) for the Sobolev spaces $H^s(\partial\Omega)$, where $\Omega\subset\R^n$ is an open bounded set with smooth boundary. For the sake of completeness, we provide a proof of this related result in Appendix~\ref{sec: proof lemma 4.1}.

\begin{remark}\label{remark: Ws-exterior-conditions} One may rescale the functions given by Lemma \ref{lemma: exterior conditions} so that the properties \ref{item 1: normalization} and \ref{item 2: estimate} holds true in the equivalent norms of $W^s(\R^n)$, defined in \eqref{eq:Gagliardonorms}. In the following proofs, we will use sequences $(\phi_N)_{N\in\N}\subset C_c^{\infty}(W)$ such that $\norm{\phi_N}_{L^2(\R^n)} \to 0$ as $N \to \infty$ and $\norm{\phi_N}_{W^s(\R^n)}=1$.
\end{remark}

\begin{lemma}[Energy concentration property]
\label{lemma: exterior determination lemma}
    Let $\Omega\subset \R^n$ be an open set which is bounded in one direction and $0<s<\min(1,n/2)$. Assume that $\gamma\in L^{\infty}(\R^n)$ with background deviations $m_{\gamma}$ satisfy $\gamma_1(x)\geq \gamma_0>0$ and $m_{\gamma}\in H^{2s,\frac{n}{2s}}(\R^n)$. If $W \subset \Omega_e$ is a nonempty, bounded, open set such that $\gamma\in C(W)$, $x_0\in W$, then for any sequence $(\phi_N)_{N \in \N}$, which satisfies the properties of Lemma~\ref{lemma: exterior conditions} there holds
    \[
        \gamma(x_0)=\lim_{N\to\infty}E_{\gamma}(\phi_N).
    \]
\end{lemma}

\begin{proof} Let $\phi_N$ be such that they have the properties given in Remark~\ref{remark: Ws-exterior-conditions} of Lemma~\ref{lemma: exterior conditions}. By the product rule for the fractional gradient~\eqref{fractional-product-rule}, 
    \begin{equation}\label{eq:product-splitting}
        |\nabla^s(\phi_N\psi)|^2\leq  2\left(|\phi_N(x)|^2|\nabla^s\psi(x,y)|^2+|\psi(y)|^2|\nabla^s\phi_N(x,y)|^2\right).
    \end{equation}
    If $\psi\in C^1_b(\R^n)$, then by the mean value theorem
    \[
    \begin{split}
        &\int_{\R^n}|\nabla^s\psi(x,y)|^2\,dy \\
        &\quad\lesssim
        \int_{B_1(x)}\frac{|\psi(x)-\psi(y)|^2}{|x-y|^{n+2s}}\,dy+\int_{\R^n \setminus B_1(x)}\frac{|\psi(x)-\psi(y)|^2}{|x-y|^{n+2s}}\,dy\\
        &\quad\lesssim \|\nabla\psi\|^2_{L^{\infty}(\R^n)}\int_{B_1(x)}\frac{dy}{|x-y|^{n+2s-2}}+\|\psi\|^2_{L^{\infty}(\R^n)}\int_{\R^n \setminus B_1(x)}\frac{dy}{|x-y|^{n+2s}}\\
        &\quad\lesssim \left(\int_{B_1(0)}\frac{dy}{|y|^{n-2(1-s)}}+\int_{\R^n\setminus B_1(0)}\frac{dy}{|y|^{n+2s}}\right)\|\psi\|_{C^1(\R^n)}^2\\
        &\quad\lesssim \|\psi\|_{C^1(\R^n)}^2
    \end{split}
    \]
    for all $x\in\R^n$. 
    By \ref{item 2: estimate} of Lemma~\ref{lemma: exterior conditions} with $t=-s$ we have $\|\phi_N\|_{L^2(\R^n)}\to 0$ as $N\to\infty$, and thus
    \begin{equation}
    \label{eq: first limit vanishs}
        \int_{\R^{2n}}|\phi_N(x)|^2|\nabla^s\psi(x,y)|^2\,dxdy\lesssim \|\psi\|_{C^1(\R^n)}^2\|\phi_N\|_{L^2(\R^n)}^2\rightarrow 0
    \end{equation}
    for all $\psi\in C_c^1(\R^n)$ as $N\to\infty$. Consider now a sequence of functions $(\eta_M)_{M\in\N}\subset C_c(\R^n)$ such that for all $M\in\N$ it holds $0\leq \eta_M\leq 1$, $\eta_M|_{Q_{1/2M}(x_0)}=1$ and $\eta_M|_{(Q_{1/M}(x_0))^c}=0$, where the cube $Q_r(x)$ is defined for all $x\in\R^n$, $r>0$ as
    \[
        Q_r(x):=\{y\in\R^n;\,|y_i-x_i|<r\quad\text{for all}\quad i=1,\ldots,n\}.
    \]
    Then for all $M\in\N$ 
    \begin{equation}\label{eq:Limsupargument}
    \begin{split}
        &\limsup_{N\to\infty}\left|\int_{\R^{2n}}(\gamma^{1/2}(y)-\gamma^{1/2}(x_0))\gamma^{1/2}(x)|\nabla^s\phi_N|^2\,dxdy\right|\\
        &\quad\leq \|\gamma\|_{L^{\infty}(\R^n)}^{1/2}\limsup_{N\to\infty}\int_{\R^{2n}}|\gamma^{1/2}(y)-\gamma^{1/2}(x_0)|\,|\nabla^s\phi_N(x,y)|^2\,dxdy\\
        &\quad=\|\gamma\|_{L^{\infty}(\R^n)}^{1/2}\limsup_{N\to\infty}\int_{\R^{2n}}|\gamma^{1/2}(y)-\gamma^{1/2}(x_0)|\,|\nabla^s(\phi_N\eta_M)(x,y)|^2\,dxdy\\
        &\quad\lesssim\|\gamma\|_{L^{\infty}(\R^n)}^{1/2}\limsup_{N\to\infty}\int_{\R^{2n}}|\gamma^{1/2}(y)-\gamma^{1/2}(x_0)|\,|\eta_M(y)|^2\,|\nabla^s\phi_N(x,y)|^2\,dxdy\\
        &\quad\leq \|\gamma\|_{L^{\infty}(\R^n)}^{1/2}\|\gamma^{1/2}-\gamma^{1/2}(x_0)\|_{L^{\infty}(Q_{1/M}(x_0))}
    \end{split}
    \end{equation}
    where we used that $\phi_N\eta_M=\phi_N$ for sufficiently large $N$, the inequality~\eqref{eq:product-splitting}, the subadditivity of limsup, the vanishing of the limit from the formula~\eqref{eq: first limit vanishs} and the property $\|\phi_N\|_{H^s(\R^n)}=1$ in the last estimate. By taking the limit $M \to \infty$ and using that $\gamma$ is continuous, we obtain that
    \begin{equation}\label{eq:limitIszero}
        \lim_{N\to\infty} \int_{\R^{2n}}(\gamma^{1/2}(y)-\gamma^{1/2}(x_0))\gamma^{1/2}(x)|\nabla^s\phi_N|^2\,dxdy = 0.
    \end{equation}
    
    We may now calculate that
    \begin{equation}
    \label{eq: reconstruction}
    \begin{split}
        \gamma(x_0)
        =&\lim_{N\to\infty}\int_{\R^{2n}}\gamma^{1/2}(x)(\gamma^{1/2}(y)-\gamma^{1/2}(x_0))|\nabla^s\phi_N|^2\,dxdy\\
        &+\gamma^{1/2}(x_0)\lim_{N\to\infty}\int_{\R^{2n}}(\gamma^{1/2}(x)-\gamma^{1/2}(x_0))|\nabla^s\phi_N|^2\,dxdy\\
        &+\gamma(x_0)\lim_{N\to\infty}\int_{\R^{2n}}|\nabla^s\phi_N|^2\,dxdy\\
        =&\lim_{N\to\infty}\int_{\R^{2n}}\gamma^{1/2}(x)(\gamma^{1/2}(y)-\gamma^{1/2}(x_0))|\nabla^s\phi_N|^2\,dxdy\\
        &+\gamma^{1/2}(x_0)\lim_{N\to\infty}\int_{\R^{2n}}\gamma^{1/2}(x)|\nabla^s\phi_N|^2\,dxdy\\
        =&\lim_{N\to\infty}\int_{\R^{2n}}\gamma^{1/2}(x)\gamma^{1/2}(y)|\nabla^s\phi_N|^2\,dxdy\\
       =&\lim_{N\to\infty}\langle \Theta_{\gamma}\nabla^s\phi_N,\nabla^s\phi_N\rangle_{L^2(\R^{2n})}\\
       =&\lim_{N\to\infty} E_\gamma(\phi_N)
    \end{split}
    \end{equation}
    where the first limit on the right hand side must vanish by \eqref{eq:limitIszero}, the second limit must vanish by the argument given in \eqref{eq:Limsupargument}, and $\lim_{N\to\infty}\int_{\R^{2n}}|\nabla^s\phi_N|^2=1$ holds as $\norm{(-\Delta)^{s/2}\phi_N}_{L^2(\R^n)}^2+\norm{\phi_N}_{L^2(\R^n)}^2 =1$ for every $N \in \N$ and $\norm{\phi_N}_{L^2(\R^n)} \to 0$ as $N\to \infty$.
\end{proof}
\begin{remark} The proof of Lemma~\ref{lemma: exterior determination lemma} shows that it is sufficient if the sequence $\phi_N$ only satisfies a weaker condition $\norm{\phi_N}_{L^2(\R^n)} \to 0$ as $N \to \infty$ instead of Lemma~\ref{lemma: exterior conditions} \ref{item 2: estimate}. This observation also carries over to the exterior determination results given in Section~\ref{subsec: exterior determination} due to the $L^2$ energy estimates of Lemma~\ref{lemma: regularity lemma} (cf. the proof of Proposition~\ref{prop: exterior determination}).
\end{remark}

\subsection{Proofs of the exterior determination results}\label{subsec: exterior determination}

Given the preparations made in Sections \ref{subsec: elliptic energy estimates} and \ref{subsec: exterior conditions}, we may now prove our results on exterior determination for the fractional conductivity equation.

\begin{proof}[Proof of Proposition~\ref{prop: exterior determination}] Let $\phi_N \in C_c^\infty(W)$ be a sequence of functions as constructed in Remark~\ref{remark: Ws-exterior-conditions} of Lemma~\ref{lemma: exterior conditions}. Let $u_N\in H^s(\R^n)$ be the unique solution to the fractional conductivity equation~\eqref{eq: sol cond eq} with exterior value $\phi_N$. We show that one can approximate the energies $E_\gamma(u_N)$ by the energies of the exterior conditions $\phi_N$.

First assume $u_f\in H^s(\R^n)$ is a solution to the fractional conductivity equation with exterior value $f \in C_c^\infty(W)$. Then we have the following energy decomposition
\begin{equation}
\label{eq: decomposition energy}
\begin{split}
    E_\gamma(u_f) &= \langle\Theta_{\gamma} \nabla^su_f,\nabla^su_f\rangle_{L^2(\R^{2n})}\\
    &= \langle\Theta_{\gamma}\nabla^s((u_f-f)+f),\nabla^s((u_f-f)+f)\rangle_{L^2(\R^{2n})} \\
&= E_\gamma(u_f-f) + 2\langle\Theta \nabla^sf,\nabla^s(u_f-f)\rangle_{L^2(\R^{2n})}+ E_\gamma(f).
    \end{split}
\end{equation}

By \ref{item 2: estimate} of Lemma~\ref{lemma: exterior conditions}, we know that $\phi_N\to 0$ in $L^2(\R^n)$ and, therefore, by Corollary~\ref{cor: regularity estimate cond eq} it follows
\begin{equation}\label{eq:convergence estimate}
    \norm{u_{N}-\phi_N}_{H^s(\R^n)} \to 0
\end{equation}
as $N\to \infty$. Using the decomposition~\eqref{eq: decomposition energy}, the boundedness of the bilinear form, the Cauchy-Schwartz inequality and the boundedness of the fractional gradient, we see that the first two terms in \eqref{eq: decomposition energy} vanish as $N\to\infty$. Thus by Lemma~\ref{lemma: exterior determination lemma} 
\begin{equation}
    \lim_{N\to\infty} E_{\gamma}(u_{N}) = \lim_{N\to\infty} E_\gamma(\phi_N) = \gamma(x_0).\qedhere\label{eq:localization}
\end{equation}
\end{proof}

Proposition~\ref{prop: exterior determination} readily implies Theorem~\ref{thm: exterior determination}. 

\begin{proof}[{Proof of Theorem~\ref{thm: exterior determination}}]
Fix $x_0 \in W$. Let $\phi_N$ be as in Proposition~\ref{prop: exterior determination}, and let $u_{N,i}$ be the solution to the conductivity equation~\eqref{eq: sol cond eq} with conductivity $\gamma_i$, $i=1,2$, and exterior condition $\phi_N$. We have by the definitions of the DN maps and Dirichlet energies that
    \[
        E_{\gamma_i}(u_{N,i})=B_{\gamma_i}(u_{N,i},u_{N,i})=B_{\gamma_i}(u_{N,i},\phi_N)=\langle \Lambda_{\gamma_i} \phi_N,\phi_N\rangle_{X^*\times X}.
    \]

Now Proposition~\ref{prop: exterior determination}, $\phi_N \in C_c^\infty(W)$ and the fact that $\Lambda_{\gamma_1}\phi_N|_W = \Lambda_{\gamma_2}\phi_N|_W$ for every $N \in \N$ let us conclude that
\begin{equation}
     \gamma_1(x_0) = \lim_{N\to \infty}E_{\gamma_1}(u_{N,1}) = \lim_{N\to \infty}E_{\gamma_2}(u_{N,2}) = \gamma_2(x_0).\qedhere
\end{equation}
\end{proof}

We conclude with a stability estimate for the exterior determination.

\begin{proof}[Proof of Proposition~\ref{prop:LipStability}]
  Proceeding as in the proof of Theorem~\ref{thm: exterior determination}, we can write
    \[
        \gamma_i(x_0)=\lim_{N\to\infty}\langle \Lambda_{\gamma_i}\phi_N,\phi_N\rangle_{X^*\times X}
    \]
    for $i=1,2$. This gives
    \[
    \begin{split}
        |\gamma_1(x_0)-\gamma_2(x_0)|&=\lim_{N\to\infty}|\langle \Lambda_{\gamma_1}\phi_N,\phi_N\rangle_{X^*\times X}-\langle \Lambda_{\gamma_2}\phi_N,\phi_N\rangle_{X^*\times X}|\\
        &=\lim_{N\to\infty}|\langle (\Lambda_{\gamma_1}-\Lambda_{\gamma_2})\phi_N,\phi_N\rangle_{X^*\times X}|\\
        &\leq \lim_{N\to\infty}\| (\Lambda_{\gamma_1}-\Lambda_{\gamma_2})\phi_N\|_{X^*}\|\phi_N\|_X\\
        &\leq \|\Lambda_{\gamma_1}-\Lambda_{\gamma_2}\|_{X\to X^*}\lim_{N\to\infty}\|\phi_N\|_{X}^2\\
        &\leq \|\Lambda_{\gamma_1}-\Lambda_{\gamma_2}\|_{X\to X^*}\lim_{N\to\infty}\|\phi_N\|_{H^s(\R^n)}^2 \\
        &\leq 2^s \|\Lambda_{\gamma_1}-\Lambda_{\gamma_2}\|_{X\to X^*} \lim_{N\to\infty}\|\phi_N\|_{W^s(\R^n)}^2
    \end{split}
    \]
    Since this estimate holds uniformly for all $x_0\in W$ and $\|\phi_N\|_{W^s(\R^n)}^2=1$ for all $N \in \N$, we have shown
    \[
        \|\gamma_1-\gamma_2\|_{L^{\infty}(W)}\leq 2^s\|\Lambda_{\gamma_1}-\Lambda_{\gamma_2}\|_{X\to X^*}.\qedhere
    \]
\end{proof}

\begin{remark} If we define the norm of $X$ with respect to the norm $W^s(\R^n)$, then the stability estimate of Proposition~\ref{prop:LipStability} holds with the constant $1$ in these modified norms.
\end{remark}

\appendix

\section{Proof of Lemma~\ref{lemma: exterior conditions}}
\label{sec: proof lemma 4.1}

\begin{proof}
    First assume $x_0=0$. Let $\psi\in C_c^{\infty}(\R)$ satisfy $\supp(\psi)\subset [-1,1]$, $\psi\neq 0$ and 
    \begin{equation}
    \label{eq: vanishing moments}
        \int_{-1}^1r^k\psi(r)\,dr=0\quad\text{for all}\quad 0\leq k\leq M-1.
    \end{equation}
    For any $N\in\N$ define
    \[
        \psi_N(x)\vcentcolon = \prod_{i=1}^n\psi(Nx_i)\in C_c^{\infty}(\R^n).
    \]
    By construction $\supp(\psi_N)\subset [-1/N,1/N]^n$, and thus $\supp(\psi_N)\to \{0\}$ as $N\to\infty$. For all $k\in\N_0$ we have
    \begin{equation}
    \label{eq: bound positive integer}
    \begin{split}
        \|\psi_N\|^2_{H^k(\R^n)}&=\sum_{|\alpha|\leq k}\|\partial^{\alpha}\psi_N\|^2_{L^2(\R^n)} =\sum_{|\alpha|\leq k}N^{2|\alpha|}\prod_{i=1}^n\|(\partial^{\alpha_i}\psi)(N\cdot)\|^2_{L^2(\R)}  \\ & =\sum_{|\alpha|\leq k}N^{2|\alpha|-n}\prod_{i=1}^n\|\partial^{\alpha_i}\psi\|^2_{L^2(\R)},
    \end{split}
    \end{equation}
and thus using the fact that  $\|\partial^{\alpha_i}\psi\|_{L^2(\R)} \leq \|\psi\|_{H^k(\R)}$ for all $i=1,...,n$ we get the upper bound
    $$\|\psi_N\|_{H^k(\R^n)}\leq C_kN^{k-n/2}\|\psi\|_{H^k(\R)}^{n} \leq C_kN^{k-n/2}.$$
A lower bound can be obtained by estimating the sum from below by the term corresponding to the multi-index $\alpha = (0,...,0,k)$:
    \[
        \|\psi_N\|^2_{H^k(\R^n)}\geq N^{2k-n}\|\psi\|^{2(n-1)}_{L^2(\R)}\|\partial^k\psi\|_{L^2(\R)}^2>0,
    \]
    where the strict positivity follows from the Poincar\'e inequality and $\psi\neq 0$. This shows that for all $k\in\N_0$ there are constants $C_k, C'_k>0$ verifying the estimate
    \begin{equation}
    \label{eq: bounds positive integers}
        C'_kN^{k-n/2}\leq \|\psi_N\|_{H^k(\R^n)}\leq C_kN^{k-n/2}
    \end{equation}
    
    Next, we wish to extend the estimate~\eqref{eq: bounds positive integers} from $k\in\N_0$ to some negative integers. Since
    \[
    \begin{split}
        \|\psi_N\|_{H^{-k}(\R^n)}&=\sup_{v\in H^k(\R^n)\vcentcolon\,\|v\|_{H^k(\R^n)}=1}|\langle\psi_N,v\rangle|\\
        &\geq |\langle \psi_N,\psi_N/\|\psi_N\|_{H^k(\R^n)}\rangle|=\frac{\|\psi_N\|_{L^2(\R^n)}^2}{\|\psi_N\|_{H^k(\R^n)}},
    \end{split}
    \]
    applying the formula~\eqref{eq: bounds positive integers} twice we immediately obtain the lower bound
    $$\|\psi_N\|_{H^{-k}(\R^n)} \geq C'_{-k}N^{-k-n/2}.$$
    
    We now produce an upper bound. Since $\psi$ is smooth and compactly supported, by the Paley--Wiener theorem $\hat\psi$ has an entire extension, and thus we can write
    $$ \hat\psi(z)= \sum_{j=0}^\infty \frac{z^j}{j!}\partial^j_{\xi}\hat\psi(0)$$
    for all $z \in \C$. However, \eqref{eq: vanishing moments} implies that $\partial^j_{\xi}\hat\psi(0)=0$ for all $j\in[0,M-1]$, and therefore $$\hat\psi(z) = z^M \sum_{j=0}^\infty \frac{z^j}{(j+M)!}\partial^{j+M}_{\xi}\hat\psi(0).$$
    One sees immediately that the function of $z$ defined by the series 
    \begin{equation}\label{eq:phiSeries}
       \hat{\phi}(z)= \sum_{j=0}^\infty \frac{z^j}{(j+M)!}\partial^{j+M}_{\xi}\hat\psi(0)
    \end{equation}
    is a meromorphic function in $\C \setminus \{0\}$ with a pole at the origin. Since this pole is removable, $\hat{\phi}$ extends into an entire function satisfying the assumptions of the Paley--Wiener theorem whose restriction to the real line (i.e.~the Fourier transform of $\phi$) decays faster than any polynomial. Therefore, there exists $\phi \in C_c^\infty(\R^n)$ with $\supp(\phi) \subset [-1,1]$ such that $\psi = (-i)^M\partial^M_r\phi$. In the following computations, we simply redefine $\phi$ so that the constant term $(-i)^M$ is included into it.
    
    Consider now the multi-index $\alpha=(k,0,...,0)$ of length $k\in (0,M]$. One can write
    \begin{align*}
        \psi_N(x) &= \prod_{i=1}^n \psi(Nx_i) = \psi(Nx_1)\prod_{i=2}^n\psi(Nx_i) \\&= N^{-k}\partial_{x_1}^k((\partial^{M-k}\phi)(Nx_1))\prod_{i=2}^n\psi(Nx_i) \\ & = N^{-k}D^\alpha_x\left( (\partial^{M-k}\phi)(Nx_1)\prod_{i=2}^n\psi(Nx_i) \right),
    \end{align*}
    which now ensures
    \begin{align*}
        \|\psi_N\|_{H^{-k}(\mathbb R^n)} &\leq N^{-k}\| (\partial^{M-k}\phi)(Nx_1)\prod_{i=2}^n\psi(Nx_i) \|_{L^2(\mathbb R^n)} \\ & = N^{-k}\|(\partial^{M-k}\phi)(N\cdot)\|_{L^2(\R)}\|\psi(N\cdot)\|^{n-1}_{L^2(\R)}\\ &= N^{-k-n/2} \|\partial^{M-k}\phi\|_{L^2(\R)}\|\psi\|^{n-1}_{L^2(\R)}
    \end{align*}
    and eventually  
    $   \|\psi_N\|_{H^{-k}(\R^n)}\leq C_{-k}N^{-k-n/2}$
    for some $C_{-k}>0$ and $0<k\leq M$. Here we have used the characterization of the negative order Sobolev space $H^{-k}(\mathbb R^n)$ as the space of distributions which can be written in the form $f=\sum_{|\alpha|\leq k} D^\alpha f_\alpha$ with $f_\alpha\in L^2(\mathbb R^n)$, equipped with the norm $\|f\|_{H^{-k}(\R^n)} = \inf\{ \sum_{|\alpha|\leq k} \|f_\alpha\|_{L^2(\R^n)}\}$, the infimum being taken over all such possible expressions of $f$.  Hence
    \begin{equation}
    \label{eq: bounds negative order spaces}
        C'_{-k}N^{-k-n/2}\leq \|\psi_N\|_{H^{-k}(\R^n)}\leq C_{-k} N^{-k-n/2}
    \end{equation}
    for some constants $C_{-k},C'_{-k}>0$ and $0<k\leq M$. 
    
    Next, we seek to the extend estimates~\eqref{eq: bounds positive integers} and \eqref{eq: bounds negative order spaces} to all sufficiently large real numbers. Let $k\in \Z$ be such that $t+s\in(k,k+1)$. Then there exists $\theta \in (0,1)$ such that $t+s=\theta k+(1-\theta)(k+1)$. Using interpolation in Bessel potential spaces (cf.~\cite[Theorem 6.4.5]{Interpolation-spaces}), \eqref{eq: bounds positive integers} and \eqref{eq: bounds negative order spaces}, we deduce
    \begin{equation}
    \label{eq: upper bound}
        \|\psi_N\|_{H^{t+s}(\R^n)}\leq C_{t,s}\left(N^{k-n/2}\right)^{\theta}\left(N^{k+1-n/2}\right)^{1-\theta}=C_{t,s}N^{t+s-n/2}
    \end{equation}
    for all $t\in\R$ such that $t+s\geq -M$, which gives the wanted upper bound. For the lower bound, we shall divide the cases $t+s\geq 0$ and $t+s<0$. 
    
    Assume first $t+s\geq 0$. Using $\|u\|_{H^{t+s}(\R^n)}\sim \|u\|_{L^2(\R^n)}+\|(-\Delta)^{\frac{t+s}{2}}u\|_{L^2(\R^n)}$ and $\widehat{\psi_N}(\xi)=N^{-n}\widehat{\psi_1}(\xi/N)$ we obtain
    \begin{equation}
    \label{eq: lower bound positive t + s}
    \begin{split}
     \|\psi_N\|_{H^{t+s}(\R^n)}&\geq C_{t,s}\|(-\Delta)^{\frac{t+s}{2}}\psi_N\|_{L^2(\mathbb R^n)}=C_{t,s}\||\xi|^{t+s}\widehat{\psi_N}\|_{L^2(\R^n)}\\
     &=C_{t,s}N^{-n}\||\xi|^{t+s}\widehat{\psi_1}(\cdot/N)\|_{L^2(\R^n)}\\
     &=C_{t,s}N^{t+s-n/2}\|(-\Delta)^{\frac{t+s}{2}}\psi_1\|_{L^2(\R^n)}.
    \end{split}
    \end{equation}
    If instead $t+s<0$, we proceed as in the estimate for the lower bound of the $\|\cdot\|_{H^{-k}(\R^n)}$-norm, this time making use of \eqref{eq: upper bound}:
    \begin{equation}
    \label{eq: lower bound negativ t + s}
    \begin{split}
        \|\psi_N\|_{H^{t+s}(\R^n)}&=\|\psi_N\|_{(H^{-(t+s)}(\R^n))^*}\geq \frac{\|\psi_N\|^2_{L^2(\R^n)}}{\|\psi_N\|_{H^{-(t+s)}(\R^n)}}\geq C_{t,s}N^{t+s-n/2}.
    \end{split}
    \end{equation} 
    Combining \eqref{eq: upper bound}, \eqref{eq: lower bound positive t + s} and \eqref{eq: lower bound negativ t + s}, we finally get the wanted estimate
     \begin{equation}\label{final-estimate-sequence-lemma}
        C'_{t,s}N^{t+s-n/2}\leq \|\psi_N\|_{H^{t+s}(\R^n)}\leq C_{t,s} N^{t+s-n/2},
    \end{equation}
    holding for all $t\in\R$ such that $t+s\geq -M$.
    
    Using our sequence $(\psi_N)_{N\in\N}$ we will construct a new sequence $(\Phi_N)_{N\in\N}$ verifying the statement of the Lemma with $x_0=0$. Observe that we have $N^{s-n/2}\lesssim \|\psi_N\|_{H^s(\mathbb R^n)}\lesssim N^{s-n/2}$. Thus we can define
    \[
        \Phi_N:=\frac{\psi_N}{\|\psi_N\|_{H^s(\R^n)}}
    \]
    for all $N\in\N$. Trivially, we have $(\Phi_N)_{N\in\N}\in C^{\infty}_c(\R^n)$, $\supp(\Phi_N)\to \{0\}$ as $N\to\infty$ and 
    \begin{equation}
    \label{eq: normalization phi}
        \|\Phi_N\|_{H^s(\R^n)}=1.   
    \end{equation}
    Morover, from \eqref{final-estimate-sequence-lemma} we get the estimates
    \begin{equation}
    \label{eq: upper bound s t norm}
        \|\Phi_N\|_{H^{t+s}(\R^n)}=\frac{\|\psi_N\|_{H^{t+s}(\R^n)}}{\|\psi_N\|_{H^{s}(\R^n)}}  \leq C_{t,s}N^{t+s-n/2}C_s^{-1}N^{-s+n/2}\leq C_{t,s}N^t
    \end{equation}
    and
    \begin{equation}
    \label{eq: lower bound 2}
        \|\Phi_N\|_{H^{t+s}(\R^n)}=\frac{\|\psi_N\|_{H^{t+s}(\R^n)}}{\|\psi_N\|_{H^{s}(\R^n)}}  \geq C_sN^{-s+n/2}C'_{t,s}N^{t+s-n/2}=C'_{t,s}N^{t}
    \end{equation}
    for all $t+s\geq -M$. 
    
    Finally, assume $x_0\in\Omega_e$ is an arbitrary point, and let $N_0$ be the minimum element of $\N$ such that $x_0+[-1/N,1/N]^n\subset\Omega_e$. Define $\phi_N := \Phi_{N+N_0}(\cdot-x_0)$ for all $N\in\N$. It is clear by construction that $\phi_N\in C^\infty_c(\Omega_e)$, with $$\mbox{supp}(\phi_N) \subset x_0 + \left[-\frac{1}{N+N_0},\frac{1}{N+N_0}\right]^n.$$
    Thus the new sequence $(\phi_N)_{N\in\N}$ verifies \ref{item 3: support}. Moreover, as the $H^t(\R^n)$ norms are translation invariant, we deduce that \ref{item 1: normalization} and \ref{item 2: estimate} hold as well.
\end{proof}

\bibliography{refs-new} 

\bibliographystyle{alpha}

\end{document}